\documentclass[letter,11pt]{amsart}
\usepackage[toc,page]{appendix}
\usepackage{amssymb} 
\usepackage{bbm}
\usepackage{graphicx}
\usepackage{color}
\usepackage[font=footnotesize]{caption}
\usepackage{mathtools}
\usepackage{pinlabel}
\usepackage[shortlabels]{enumitem}
\usepackage{graphicx}
\graphicspath{ {images/} }
\usepackage{comment}
\usepackage{tikz-cd}
\usepackage{subfigure}
\usepackage{geometry}
\usepackage{mathrsfs}

\usepackage{hyperref}

\usepackage{float}
\usepackage{calc}
\def\thetitle{{ . }}
\hypersetup{
	pdftitle=  \thetitle,
	pdfauthor=  {}  
}

\newtheorem{THM}{Theorem}[section]

\newtheorem*{THM*}{Theorem}
\newtheorem{LEM}[THM]{Lemma}
\newtheorem{COR}[THM]{Corollary}
\newtheorem{prop}[THM]{Proposition}
\newtheorem{que}{Question}
\newtheorem{conj}[que]{Conjecture}
\newtheorem{defn}[THM]{Definition}
\newtheorem{conv}[THM]{Convention}

\newtheorem{nota}[THM]{Notation}
\newtheorem{cons}[THM]{Construction}

\newtheorem{ep}[THM]{Example}
\newtheorem*{main (a)}{Theorem \ref{main} (a)}
\newtheorem*{main (b)}{Theorem \ref{main} (b)}

\theoremstyle{remark}

\newtheorem{rmk}[THM]{\textbf{\emph{Remark}}}

\usepackage{tikz}
\usetikzlibrary{calc,decorations.markings}
\usetikzlibrary{shapes,snakes}

\theoremstyle{definition}
\newtheorem*{defn*}{Definition}

\newcommand\N{\mathbb{N}}
\newcommand\Q{\mathbb{Q}}
\newcommand\R{\mathbb{R}}

\newcommand\Z{\mathbb{Z}}

\newcommand{\F}{\mathcal{F}}

\newcommand{\Fs}{\mathcal{F}^{s}}
\newcommand{\Fu}{\mathcal{F}^{u}}
\newcommand{\Fws}{\mathcal{F}^{ws}}
\newcommand{\Fwu}{\mathcal{F}^{wu}}

\newcommand{\Ews}{\mathcal{E}^{ws}}
\newcommand{\Ewu}{\mathcal{E}^{wu}}
\newcommand{\Ors}{\mathcal{O}^{s}}
\newcommand{\Oru}{\mathcal{O}^{u}}

\newcommand{\G}{\mathcal{G}}

\newcommand{\Ll}{\mathcal{L}}

\newcommand{\Lle}{\stackrel{L(\F)}{<}}

\newcommand{\Or}{\mathcal{O}}
\newcommand{\w}{\widetilde}
\newcommand{\wh}{\widehat}

\newcommand{\Int}{\text{Int}}

\newcommand{\nls}{\text{NLS}}
\newcommand{\lo}{\text{LO}}
\newcommand{\ctf}{\text{CTF}}
\newcommand{\In}{\mathcal{I}_\nls}
\newcommand{\Il}{\mathcal{I}_\lo}
\newcommand{\Ic}{\mathcal{I}_\ctf}
\newcommand{\IR}{\mathcal{I}_\R}

\newcommand{\Home}{\text{Homeo}_+}
\newcommand{\s}{\mathbf{s}}

\newcommand\SL{\operatorname{SL}}

\begin{document}


\title{Left-orderability in Dehn fillings of pseudo-Anosov mapping tori}
\author{Bojun Zhao$^1$}
\thanks{$^1$D\'epartement de math\'ematiques, Universit\'e du Qu\'ebec \`a Montr\'eal, 201 President Kennedy Avenue, Montréal, QC, Canada H2X 3Y7. 
Email: \href{mailto:bjzhaotopology@gmail.com}{bjzhaotopology@gmail.com}}
\maketitle

\begin{abstract}
For pseudo-Anosov mapping tori with co-orientable invariant foliations
and monodromies reversing their co-orientations,
a family of taut foliations was constructed in previous work
on Dehn fillings with all rational slopes outside a neighborhood
of the degeneracy slope.
In this paper, we prove that all such Dehn fillings
have left-orderable fundamental groups.
We present two approaches,
both establishing left-orderability through the branching behavior
of taut foliations.
The first approach produces an $\R$-covered foliation
arising from this family for each filling slope,
and the second approach shows that, 
depending on the choice of a suitable system of arcs on $\Sigma$,
the resulting foliation either has one-sided branching or is $\R$-covered.
Consequently, the second approach associates to each Dehn filling
a family of representations of its fundamental group into $\G_\infty$,
the group of germs at infinity,
whereas the first approach yields an explicit left-invariant order.
As an application,
combining our results with earlier work in the literature,
we verify the L-space conjecture for all surgeries
on the $(-2,3,2q+1)$-pretzel knot ($q \geqslant 3$) in $S^3$.
From another perspective, 
$\mathbb{R}$-covered foliations can be produced systematically
across a large family of 
Dehn fillings on cusped hyperbolic manifolds, 
and in some cusped manifolds they cover all fillings that 
admit co-orientable taut foliations. 
This expands the class of known $\mathbb{R}$-covered foliations.
\end{abstract}


\section{Introduction}

Throughout this paper,
all $3$-manifolds are orientable,
and all taut foliations are co-orientable.

A central program in $3$-dimensional topology is to establish connections
between Floer theory and various algebraic and geometric
structures of $3$-manifolds.
In this direction, the L-space conjecture was formulated as a guiding principle,
suggesting a unified way to organize $3$-manifolds 
across these different perspectives.

\begin{conj}[L-space conjecture, {\cite{BGW13, Juhasz15}}]
Let $M$ be a closed, orientable, irreducible $3$-manifold. 
Then the following statements are equivalent.

(1)
$M$ is a non-L-space.

(2)
$\pi_1(M)$ is left-orderable.

(3)
$M$ admits a co-orientable taut foliation.
\end{conj}

It was shown by Gabai~\cite{G83} that statement (3) holds for manifolds $M$
with positive first Betti number, and by Boyer-Rolfsen-Wiest~\cite{BRW05} that 
statement (2) also holds in this case.
Moreover, Ozsv\'ath and Szab\'o~\cite{OS04} established that statement (3)
implies statement (1)
(see also \cite{Bowden16, KazezR17}).

Geometric interpretations of left-orderability remain mysterious in general.
However, one of the most direct manifestations comes from
the natural dynamical realizations associated with certain taut foliations.
Recall that taut foliations in closed $3$-manifolds fall into three types
according to their branching behavior:
$\R$-covered, one-sided branching, and two-sided branching.
$\mathbb{R}$-covered foliations serve as an idealized model:
they form an intermediate class between fibrations and general taut foliations,
with a global transverse coordinate in the universal cover while allowing
nontrivial holonomy along the leaves;
foliations with one-sided branching can be viewed as a partially idealized type.
These two types are closely related to the underlying geometry of 
the manifold and the dynamics of the associated flows
\cite{Thu97, Cal00, Fen02, Cal03, Fen12}.
In the context of left-orderability, 
these two types play particularly important roles.
Every co-orientable $\R$-covered foliation induces a faithful $\pi_1$-action on $\R$,
giving rise to an explicit left-invariant order
(see \cite[Theorem 7.10]{CD03}, \cite[Theorem 2.4]{BRW05});
every co-orientable taut foliation with one-sided branching induces 
a faithful $\pi_1$-representation into $\G_\infty$,
the group of germs at infinity,
implying that the ambient $3$-manifold has left-orderable fundamental group~\cite{Z25},
although no explicit left-invariant order is currently known
(compare with \cite[Extension vs.\ realization]{M15}).

The following question, 
proposed by Brittenham-Naimi-Roberts for hyperbolic manifolds~\cite[p. 466]{BrittenhamNR97} 
and by Calegari~\cite[Question 8.3]{Cal02} for atoroidal manifolds, 
asks whether left-orderability and taut foliations can be linked via
$\R$-covered foliations.

\begin{que}\label{que: R-covered exist?}
Let $M$ be an atoroidal $3$-manifold admitting a taut foliation. 
Does $M$ necessarily admit an $\R$-covered foliation?
\end{que}

Some counterexamples were found by Brittenham~\cite{Bri02} among graph manifolds.
In hyperbolic manifolds,
existing examples of $\R$-covered foliations
are typically tied to specific geometric or dynamical features; 
beyond such settings, 
their existence in general remains poorly understood.
Indeed, if a co-orientable taut foliation contains a genuine sublamination,
it must have two-sided branching.
In contrast, when a taut foliation contains no genuine sublamination,
its branching behavior typically displays no obvious features
at the level of the underlying $3$-manifold,
and is therefore difficult to detect in general.
See Subsection~\ref{subsec: at most one-sided branching} for a summary 
in this direction.

\subsection{Left-orderability and Dehn fillings on cusped manifolds}

The study of the L-space conjecture is naturally related to 
Dehn surgeries on knots and links.
For a knot in $S^3$,
properties of the knot itself determine whether 
nontrivial L-space surgeries exist
and which slopes realize them,
with the L-space slopes forming a uniform interval~\cite{OS05, KMOS07, OS10}.
A similar phenomenon also appears for knots
in closed $3$-manifolds \cite{RasmussenR17}.
This motivates a further investigation of surgery slopes
on knots and links in relation to the three aspects of the L-space conjecture:
to determine which slopes yield (1), (2), or (3) via Dehn surgery,
and to understand how these structures arise in the resulting $3$-manifolds.

Left-orderability, as a group-theoretic property,
can be established for Dehn fillings through 
a variety of geometric and topological structures,
including those coming from
$\SL_2(\R)$-character varieties 
(see \cite{CullerD18}, and for example \cite{Gao23, HeraldZ19}),
$\text{SU}(2)$-character varieties \cite{DunfieldR25},
Euler classes associated to taut foliations \cite{BoyerH19, Hu23},
properties of pseudo-Anosov flows \cite{Zung24},
and left-orderable slope-detections~\cite{BGH25}.
In this paper, we derive the left-orderability of
the resulting Dehn fillings from $\R$-covered foliations
or foliations with one-sided branching.



Since all L-space knots in $S^3$ are fibered~\cite{Ni07, Ghiggini08},
it is natural to consider Dehn fillings
of pseudo-Anosov mapping tori
in the context of the L-space conjecture.

\begin{conv}\rm
Let $\varphi: \Sigma \to \Sigma$ be
an orientation-preserving pseudo-Anosov homeomorphism
on a compact orientable surface $\Sigma$
with $\partial \Sigma \ne \emptyset$ and $\chi(\Sigma) < 0$.
Let $M = \Sigma \times I / \stackrel{\varphi}{\sim}$
be the mapping torus of $\varphi$,
that is, the quotient of $\Sigma \times [0,1]$
with respect to the equivalence relation
\[(x,1) \sim (\varphi(x),0), \quad \forall\ x \in \Sigma.\]
Let $\Fs$ and $\Fu$ denote the stable and unstable foliations of $\varphi$ on $\Sigma$.
\end{conv}

We adopt the canonical meridian-longitude coordinate system
for each boundary component following \cite{R01},
which, except for a special case,
coincides with the standard meridian-longitude coordinates
when $M$ is the exterior of a fibered knot in $S^3$.
On a component $T$ of $\partial M$,
all slopes represented by essential simple closed curves on $T$
are parametrized by $\Q \cup \{\infty\}$,
where $\infty=\frac{1}{0}$ by convention.
These slopes are referred to as the \emph{rational slopes} on $T$.

Let $\psi$ denote the suspension flow of $\varphi$ in $M$.
For each boundary component $T$ of $M$,
the set of closed orbits of $\psi$ on $T$
is a union of $2n$ parallel essential simple closed curves
for some $n \in \N_+$,
whose common slope is referred to as the \emph{degeneracy slope} of $T$
and is denoted by $\delta_T$.
The \emph{degeneracy locus} $d(T)$ of $\psi$ on $T$
is a local system identified with an integer pair $(n u; n v)$,
where $\gcd(u,v)=1$, $u > 0$
and $\delta_T = \frac{u}{v} \in \Q \cup \{\infty\}$
in the canonical meridian-longitude coordinates
(with $u=1, v=0$ if $\delta_T=\infty$).
See Definition \ref{def: degeneracy locus} for details.

Suppose that the stable foliation $\Fs$ is co-orientable. 
Then $\Fu$ is also co-orientable, 
and $\varphi$ either preserves or reverses both co-orientations.
When the monodromy $\varphi$ preserves the co-orientations,
it was shown implicitly in \cite{Gabai92} by Gabai that the resulting Dehn filling
admits a co-orientable taut foliation
whenever the filling slope is distinct from the degeneracy slope
on each boundary component.
Building on these foliations, Zung~\cite{Zung24} proved that such a Dehn filling
has left-orderable fundamental group
if the filling slopes have the same sign
with respect to the slope coordinates
in which $\delta_T$ and $\frac{0}{1}$
form an ordered basis on each boundary component $T$.

When $\Sigma$ has connected boundary,
Gabai's work~\cite{Gabai92} implies that $M$ admits at most one Dehn filling
with no co-orientable taut foliation,
and consequently at most one L-space Dehn filling~\cite{OS04}.
The situation is quite different when 
$\varphi$ reverses the co-orientation of $\Fs$. 
In this case, $M$ can be Floer simple, 
so that all rational slopes in some open interval of $\R P^1$ yield
L-space Dehn fillings~\cite{RasmussenR17} 
(see Figure \ref{RP^1} (a) for an illustration of L-space filling slopes).
This is the case on which we focus in this paper;
see Subsection~\ref{subsec: applications}
for several families of examples.

\begin{conv}\rm\label{convention on filling multislopes}
Let $T_1, \ldots, T_r$ denote the boundary components of $M$,
and we denote by $(p_i;q_i)$ the degeneracy locus on $T_i$.
For any $s_1,\ldots,s_r \in \Q \cup \{\infty\}$,
we denote by $M(s_1,\ldots,s_r)$
the Dehn filling of $M$ along $\partial M$
with slope $s_i$ on $T_i$.
For each $1 \leqslant i \leqslant r$,
choose a boundary component $C_i \subseteq \partial \Sigma$ with
$C_i \times \{0\} \subseteq T_i$,
and define $c_i$ to be the order of $C_i$ under $\varphi$, namely
\[c_i = \min \{ k \in \mathbb{N}_+ \mid \varphi^{k}(C_i) = C_i \}.\]
\end{conv}

The following theorem was proved in \cite{Z26}.

\begin{figure}
	\centering
	\subfigure[]{
	\includegraphics[width=0.3\textwidth]{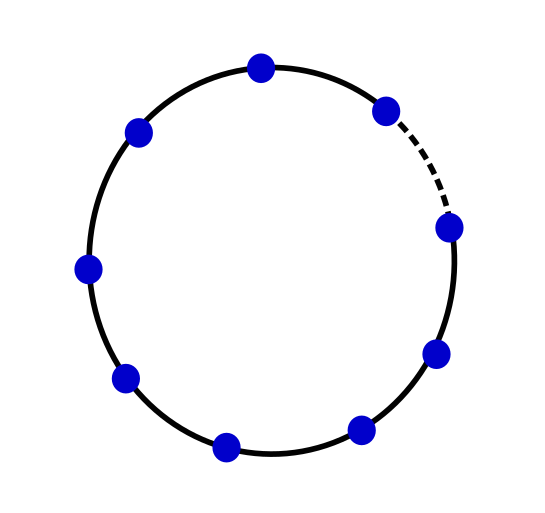}}
\subfigure[]{
	\includegraphics[width=0.3\textwidth]{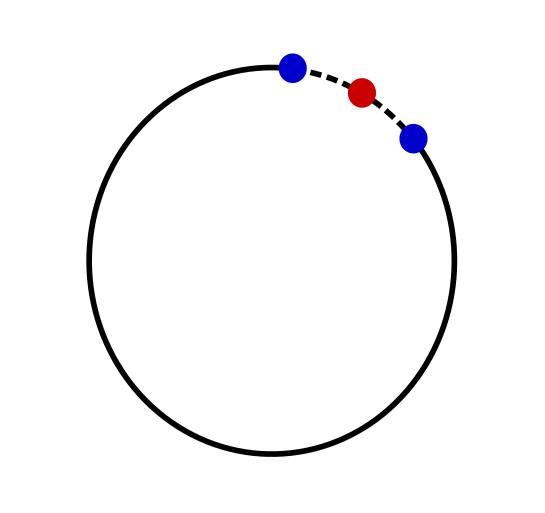}}
	\caption{For a Floer simple knot manifold,
there exists a finite set $P$ determined by the Turaev torsion
such that a rational slope yields an L-space Dehn filling
if and only if it lies in the closure of a component of
$(\R \cup \{\infty\}) - P$ \cite{RasmussenR17}.
In (a), the set $P$ is shown as blue dots.
The dashed segment represents the closed interval of L-space slopes,
and every rational slope in its complement (the solid segment)
yields a non-L-space filling.
In (b), we illustrate the slopes appearing in Theorem~\ref{foliation}.
The degeneracy slope $\frac{p_i}{q_i}$ is shown as the red dot,
the two bounds $\frac{p_i}{q_i \pm c_i}$ as blue dots,
the interval $J_i$ as the solid segment,
and the remaining neighborhood of the degeneracy slope
as the dashed segment.}\label{RP^1}
\end{figure}

\begin{THM}[{\cite{Z26}}]\label{foliation}
Suppose that $\Fs$ is co-orientable and $\varphi$ reverses its co-orientation.
For each $1 \leqslant i \leqslant r$, let $J_i$ be the open interval in
$\R \cup \{\infty\} \cong \R P^{1}$ between
$\frac{p_i}{q_i + c_i}$ and $\frac{p_i}{q_i - c_i}$
which does not contain $\frac{p_i}{q_i}$.
Fix a slope $s_i \in J_i \cap (\Q \cup \{\infty\})$ for each $i$,
and let $\s$ denote the multislope $(s_1,\ldots,s_r)$.
Then the Dehn filling $M(\s)$ admits a co-orientable taut foliation.
\end{THM}

See Figure \ref{RP^1} (b) for an illustration of the interval $J_i$.
Throughout the introduction, 
we use $J_i$ to denote this interval of slopes on $T_i$.


An \emph{admissible system of arcs} with respect to $(\Sigma,\varphi)$ is 
a family of disjoint properly embedded oriented arcs 
positively transverse to $\Fs$ and disjoint from its singularities,
with certain additional conditions; 
see Definition \ref{system of arcs} for details.
Under the assumptions of Theorem~\ref{foliation},
each admissible system of arcs $\alpha$ with respect to $(\Sigma,\varphi)$
determines a co-orientable taut foliation $\F_\alpha(\s)$ of $M(\s)$;
see Proposition \ref{prop: construction from the new branched surface} 
for the correspondence.
The main result of this paper is the following theorem.

\begin{THM}\label{main}
Suppose that $\Fs$ is co-orientable and $\varphi$ reverses its co-orientation.
Fix a slope $s_i \in J_i \cap (\Q \cup \{\infty\})$ for each
$1 \leqslant i \leqslant r$,
and let $\s$ denote the multislope $(s_1,\ldots,s_r)$.

\textnormal{(a)}
There exists an admissible system of arcs
$\alpha^*$ with respect to $(\Sigma,\varphi)$
such that the induced foliation $\F_{\alpha^*}(\s)$ is $\R$-covered.

\textnormal{(b)}
For any admissible system of arcs
$\alpha$ with respect to $(\Sigma,\varphi)$,
the foliation $\F_\alpha(\s)$
either has one-sided branching or is $\R$-covered.
\end{THM}

In both (a) and (b),
the resulting foliation on $M(\s)$ establishes
the left-orderability of $\pi_1(M(\s))$.
The $\R$-covered foliation produced in (a)
gives rise to an explicit left-invariant order.
More generally, 
(b) implies that every foliation
$\F_\alpha(\s)$
induces a faithful representation into the group of germs at infinity.
Consequently, we obtain the following corollary.

\begin{COR}
Fix a slope $s_i \in J_i \cap (\Q \cup \{\infty\})$ for each
$1 \leqslant i \leqslant r$,
and let $\s$ denote the multislope $(s_1,\ldots,s_r)$.

(a)
We can construct an explicit left-invariant order of $\pi_1(M(\s))$ from 
the $\R$-covered foliation in Theorem \ref{main} (a).

(b)
For any admissible system of arcs with respect to $(\Sigma,\varphi)$,
there exists an induced faithful representation
\[\pi_1(M(\s)) \to \G_\infty.\]
\end{COR}

Combining Theorem~\ref{main} with Theorem~\ref{foliation} and \cite{OS04},
we have the following summary.
Although many Dehn fillings of $M$ may be L-spaces,
we can always produce $3$-manifolds satisfying
the three conditions of the L-space conjecture simultaneously,
whenever the filling slope on every boundary component lies in 
the range outside a neighborhood of the degeneracy slope.

\begin{COR}
Let $s_i$ be a slope on $T_i$ contained in
$J_i \cap (\Q \cup \{\infty\})$ for each $1 \leqslant i \leqslant r$,
and let $\s$ denote the multislope $(s_1,\ldots,s_r)$.
Then $M(\s)$ is a non-L-space that admits a co-orientable taut foliation
and has left-orderable fundamental group.
In particular,
$M(\s)$ admits a co-orientable $\R$-covered foliation.
\end{COR}

\subsection{Applications in specific knots}\label{subsec: applications}

We offer some applications in this subsection.
For a manifold $N$, which is a knot manifold, a link manifold, or a cusped
hyperbolic $3$-manifold, we denote by $\In(N), \Il(N), \Ic(N)$ the sets
of rational slopes yielding Dehn fillings 
which are non-$L$-spaces, 
have left-orderable fundamental groups, or admit co-orientable taut foliations,
respectively.

Suppose that $\Sigma$ has genus one and connected boundary. 
Note that any pseudo-Anosov mapping class on $\Int(\Sigma)$ corresponds to
a hyperbolic element of $\mathrm{SL}(2,\Z)$~\cite[page 54]{FarbM11},
and the associated linear Anosov automorphism of the torus determines
stable and unstable foliations given by curves without singularities.
Hence $\Fs$ is non-singular and co-orientable.

Let $\delta$ denote the degeneracy slope of $\psi$ on $\partial M$.
If $\delta=\infty$, then $\varphi$ must preserve the co-orientation on $\Fs$,
and hence all Dehn fillings with slopes except $\infty$ admit
a co-orientable taut foliation by Gabai~\cite{Gabai92}
(later by Roberts~\cite{R00, R01} via a different construction),
and also have left-orderable fundamental group by Zung~\cite{Zung24}.

Now assume otherwise.
Up to a choice of orientation, we may assume that $\delta$ is positive.
It was already known from the results of Roberts in 2001~\cite{R00, R01} that
$M(s)$ admits a co-orientable taut foliation for all
$s\in(-\infty,1]\cap\Q$.
In this case, $\varphi$ must reverse the co-orientation on $\Fs$,
so Theorem~\ref{main} applies.
This establishes the left-orderability of these manifolds.

\begin{prop}\label{prop: once-puncture}
    Suppose that $\Sigma$ has genus one and a unique boundary component,
    and that the degeneracy slope on $\partial M$ is positive.
    Then all rational slopes in $(-\infty, 1]$ are contained in 
    $\In(M), \Il(M), \Ic(M)$ simultaneously.
\end{prop}

All genus one fibered Floer simple knot manifolds are included in 
Proposition \ref{prop: once-puncture}.
For example,
the cusped hyperbolic manifold $m003$ satisfies the assumptions
of the proposition and has
$\In(m003) = (-\infty,1] \cap \Q$
under the canonical coordinate system \cite{Dun20b, Dun}.
Proposition~\ref{prop: once-puncture} therefore implies that
all non-L-space Dehn fillings have left-orderable fundamental group.
We refer the reader to \cite{Baker11} for further examples
of Floer simple manifolds satisfying the assumptions of 
Proposition \ref{prop: once-puncture}.

Clearly, for any Anosov homeomorphism on a closed torus associated with
a hyperbolic matrix of determinant $-1$,
Theorem~\ref{main} also applies to the corresponding Dehn surgeries along
any collection of periodic orbits.


The family of $(-2,3,2q+1)$-pretzel knots in $S^3$ (with $q \in \Z_{\geqslant 3}$)
is a collection of hyperbolic L-space knots,
which have served as classical examples
in the study of exceptional Dehn surgeries
to lens spaces and more general L-spaces.
It was found by Fintushel and Stern in 1980~\cite{FintushelS80} that
the $(-2,3,7)$-pretzel knot admits two distinct lens space surgeries,
and by Bleiler and Hodgson in 1996~\cite{BleilerH96} that
the $(-2,3,9)$-pretzel knot admits two nontrivial finite surgeries
with non-cyclic fundamental group.
In 2005, Ozsv\'ath and Szab\'o~\cite[page~1291]{OS05} proved
that each knot in this family is an L-space knot.
We refer the reader to \cite[Theorem~6.7]{Gabai86} for 
the fiberedness of each
$(-2,3,2q+1)$-pretzel knot, and to \cite[Corollary~5]{Oertel84} for their
hyperbolicity.

As predicted by the L-space conjecture,
for each $(-2,3,2q+1)$-pretzel knot $K$ in $S^3$ with
$q\in\Z_{\geqslant 3}$,
the sets $\Il(S^3-K)$ and $\Ic(S^3-K)$ should be identified with
\[\In(S^3-K)=(-\infty,2g(K)-1)\cap\Q,\]
where $g(K)$ denotes the genus of $K$, and the set of slopes
$(-\infty,2g(K)-1)\cap\Q$ is essentially known from
\cite{KMOS07, OS10}.
It was first shown by Krishna~\cite{Krishna20} that
\[\Ic(S^3-K)=(-\infty,2g(K)-1)\cap\Q.\]
Later, a different proof was obtained via Theorem~\ref{foliation}
\cite[Proposition~1.9]{Z26}.
The non-left-orderability of the manifolds surgered by rational slopes in
$[2g(K)-1,+\infty)$ was proved by Nie~\cite[Theorem~2]{Nie19}.

It therefore remains to prove that any rational slope in
$(-\infty,2g(K)-1)$ yields a surgered manifold with left-orderable
fundamental group.
Some progress has been made in this direction:
Varvarezos \cite{Varvarezos21}, 
motivated by Culler-Dunfield \cite[pages 1424-1427]{CullerD18},
showed that 
$(-\infty,6)\cap\Q\subseteq\Il(S^3-K)$ when $q=3$,
and very recently Tran~\cite{Tran25} proved that
$(-\infty,2\lfloor \frac{2q+4}{3}\rfloor)\cap\Q\subseteq\Il(S^3-K)$ when $q \ne 4$.
The methods of Varvarezos and Tran both proceed via
$\SL_2(\R)$ representations, which are very different from our approach.


Combining Theorem \ref{main} with \cite{Nie19, Krishna20, Z26},
we have the following.

\begin{prop}
    Let $q \in \Z_{\geqslant 3}$ and let $K$ be 
    the $(-2,3,2q+1)$-pretzel knot in $S^3$.
    Then \[\In(S^3-K) = \Il(S^3-K) = \Ic(S^3-K) = (-\infty, 2g(K)-1) \cap \Q.\]
\end{prop}

Moreover, Proposition 1.12 of \cite{Z26} implies the following.

\begin{prop}
Let $K$ be a hyperbolic L-space knot in $S^3$. 
If the stable foliation of its monodromy is co-orientable 
and has no singularities in the interior of the fibered surface, 
then every non-L-space obtained by Dehn surgery on $K$ 
has left-orderable fundamental group.
\end{prop}

Dunfield’s census \cite{Dun} provides many examples of cusped hyperbolic manifolds 
satisfying the hypotheses of Theorem~\ref{main}. 
In~\cite{Dun20}, these manifolds are tested for Floer simplicity, 
as well as the cones of all non-L-space filling slopes
(see~\cite{Dun20b} for detailed data).
In \cite[Examples~1.13-1.17]{Z26}, 
there are some examples which are complements of L-space knots in spherical manifolds, 
with the property $\In(N) = \Ic(N)$ verified by Theorem~\ref{foliation}. 
Theorem \ref{main} shows that all non-L-space Dehn fillings 
for such manifolds $N$ have left-orderable fundamental group.

For general cusped manifolds,
the set $\In$ can take more varied forms than for L-space knots in $S^3$;
we provide some examples below for reference.
All data on non-L-space Dehn fillings come from \cite{Dun20b},
while information on the cusped manifolds comes from \cite{Dun}.

\begin{ep}\rm
The hyperbolic cusped manifolds
$m146$, $v2585$, $m303$, $s520$, $v1206$, $t02779$, $o9_{06362}$
are complements of L-space knots in $\R P^3$
with genus $g = 3,4,5,6,7,8,9$, respectively.
Up to orientation, they have degeneracy slope $\delta = 4g-2$
and $\In = (-\infty,2g-1) \cap \Q$.
Theorem \ref{main} applies to these cusped manifolds and implies that
all non-L-space Dehn fillings have left-orderable fundamental groups.
\end{ep}

Cusped manifolds with more than one lens space Dehn filling are Floer simple.
Examples include the following.

\begin{ep}\rm
The hyperbolic cusped manifolds $m122, m280, v0751, v0173, o9_{00008}$
have three distinct lens space Dehn fillings of slopes $\delta,\infty$ and
one of $\{\delta+1, \delta-1\}$,
where $\delta$ is the degeneracy slope,
and Theorem \ref{main} applies.
For these manifolds we obtain
\[g(m122) = 2, \ \delta(m122) = 4,\
(-\infty,2) \cap \Q = \In(m122) \subseteq \Il(m122)\]
\[g(m280) = 2, \ \delta(m280) = -4,\
(-2,+\infty) \cap \Q = \In(m280) \subseteq \Il(m280)\]
\[g(v0751) = 3, \ \delta(v0751) = -6,\
(-3,+\infty) \cap \Q = \In(v0751) \subseteq \Il(v0751)\]
\[g(v0173) = 4, \ \delta(v0173) = 10,\
(-\infty,5) \cap \Q = \In(v0173) \subseteq \Il(v0173)\]
\[g(o9_{00008}) = 5, \ \delta(o9_{00008}) = 12,\
(-\infty,6) \cap \Q = \In(o9_{00008}) \subseteq \Il(o9_{00008})\]
\end{ep}

\begin{ep}\rm
The hyperbolic cusped manifold $o9_{26541}$ is the complement of an L-space knot in 
a lens space of order $87$, 
with genus $3$ and degeneracy slope $-\frac{8}{3}$ 
(where the lens space filling slope is $3$).
Theorem \ref{main} applies,
and we obtain
$(\Q \cup \{\infty\}) - (-4,-2) = \In(o9_{26541}) \subseteq \Il(o9_{26541})$.
\end{ep}

\begin{rmk}
Let $\IR(N)$ denote the set of rational slopes on a cusped manifold $N$ 
for which the corresponding Dehn fillings admit $\mathbb{R}$-covered foliations. 
Since the left-orderability can be established from
an $\mathbb{R}$-covered foliation on each filled manifold, 
it follows that $\IR(N)=\In(N)=\Ic(N)$ 
when $N$ is one of the cusped manifolds considered above.
\end{rmk}

We note that Theorem \ref{foliation} may not provide the sharp bound in all cases. 
For example, as illustrated in \cite[Example~1.18]{Z25}, 
the manifold $o9_{19364}$ is the complement of an L-space knot in $S^3$ 
with genus $14$ and degeneracy slope $48$. 
In this case, $\In(o9_{19364}) = (-\infty,27) \cap \Q$,
whereas Theorem~\ref{foliation} establishes only
$(-\infty,24] \cap \Q$.
By Theorem~\ref{main}, we obtain
$(-\infty,24] \subseteq \Il(o9_{19364})$,
while it remains open whether 
$(24,27) \cap \Q \subseteq \Il(o9_{19364})$.

\subsection{Foliations with at most one-sided branching}\label{subsec: at most one-sided branching}

As a natural model in the theory of foliations,
the study of $\R$-covered foliations began in several contexts,
including foliations transverse to Seifert fibrations~\cite{EisHN81},
Anosov flows~\cite{Fen94},
and slitherings~\cite{Thu97}.
Examples of $\R$-covered foliations that do not arise from slitherings
were later constructed in~\cite{Cal99}.
In Seifert fibered manifolds,
every foliation transverse to the Seifert fibration
is $\R$-covered~\cite{RobS99}, \cite[Lemma~5.6]{BRW05}.
In the context of Dehn surgeries on knots in $3$-manifolds,
$\R$-covered foliations were produced via
integer surgeries on the figure-eight knot in $S^3$ \cite{Fen94},
and this construction was extended in \cite{BI23} to broader classes of
Dehn surgeries with integer slopes in suitable framings.

Taut foliations with one-sided branching were first constructed
in hyperbolic $3$-manifolds by Meigniez~\cite{Meigniez91}.
See further examples in \cite[Example~5.02]{Cal03} and \cite[Example~4.43]{Cal07}.
By convention,
a taut foliation is said to have \emph{at most one-sided branching} if
it either has one-sided branching or is $\R$-covered.

In contrast,
many taut foliations are constructed from genuine essential laminations and hence
necessarily have two-sided branching
(see Proposition~\ref{genuine two-sided branching}),
including foliations in case~(2) below and
foliations in~(1) whenever the ambient manifold
is not a surface bundle.
Thus, two-sided branching is common among 
taut foliations whose branching behavior is known.

The branching behavior of taut foliations
is well-understood in the three classes listed below,
while foliations with at most one-sided branching
occur only in restricted situations.

\begin{enumerate}[(1)]
\item
A finite-depth taut foliation cannot have one-sided branching,
and it is $\R$-covered only when the ambient manifold is a surface bundle over $S^1$.

\item
If a co-orientable taut foliation is obtained by filling monkey saddles from
a genuine essential lamination whose guts consist of solid tori,
then it necessarily has two-sided branching.

\item
The stable foliation of an Anosov flow cannot have one-sided branching~\cite{Fen95}.
If such a foliation is $\R$-covered,
then it is of exactly two types:
\emph{trivial $\R$-covered} or \emph{skewed $\R$-covered}~\cite{Fen94}.
The former arises as the suspension of an Anosov homeomorphism of the torus,
while the latter is completely characterized by
extended convergence group actions on $\R$~\cite[Subsection~7.1]{Thu97}.
\end{enumerate}

Theorem \ref{main} indicates that
foliations with at most one-sided branching
occur in a substantial class of $3$-manifolds.
These manifolds arise from Dehn fillings and
span a wide range of geometric and topological types.
This suggests that such foliations are far more prevalent
than previously understood,
and may provide a useful perspective on left-orderability in 
the context of the L-space conjecture.


\subsection{Organization}

This paper is organized as follows.

Section \ref{sec: preliminaries} collects the necessary background,
with all conventions fixed in Subsection~\ref{subsec: convention}.
We then review the branching behavior of taut foliations
in Subsection \ref{subsec: branching behavior},
the material on pseudo-Anosov flows needed later
in Subsection \ref{subsec: pseudo-Anosov flow},
and the basic notions of branched surfaces 
in Subsection \ref{subsec: branched surface}.

In Section \ref{sec: branched surface}, 
we describe a foliation
$\F_\alpha$ in the mapping torus $M$ determined by 
any admissible system of arcs $\alpha$, 
together with its extensions $\F_\alpha(\s)$ to 
the corresponding Dehn fillings $M(\s)$.
In Subsection \ref{subsec: admissible system of arcs}, 
we define admissible systems of arcs and 
review the branched surface used in the construction of \cite{Z26}.
In Subsection \ref{subsec: perturb}, 
we modify this branched surface to
obtain a new branched surface $B_\alpha$,
from which the desired foliations will be produced.
In Subsection \ref{subsec: Fried's surgery}, 
we verify that $M(\s)$
admits a pseudo-Anosov flow obtained from the suspension flow by Fried's surgery,
and then summarize the constructions of $\F_\alpha$ and
$\F_\alpha(\s)$ together with their relation to the corresponding flows.

Section \ref{sec: R-covered} is devoted to the proof of
Theorem \ref{main} (a).
Under the assumptions of Theorem~\ref{main},
we first construct a specific admissible system of arcs $\alpha$
with respect to $(\Sigma, \varphi)$,
which gives rise to a foliation $\F_\alpha$ in $M$ that
extends to a co-orientable taut foliation
$\F_\alpha(\s)$ in $M(\s)$.
We then analyze the interactions between the leaves of $\F_\alpha$ and 
the weak stable foliation of the suspension flow on $M$ 
(Subsection~\ref{subsec: F and Fws}),
and finally show that $\F_\alpha(\s)$ is $\R$-covered 
(Subsection \ref{subsec: verify R-covered}).

In Section \ref{sec: one-sided branching},
we complete the proof of Theorem \ref{main} (b).
We begin with an arbitrary admissible system of arcs $\alpha$.
Subsection \ref{subsec: F and Ewu}
analyzes the intersection behavior between
the induced taut foliation $\F_\alpha(\s)$ in $M(\s)$
and the weak unstable foliation of the pseudo-Anosov flow on $M(\s)$
produced by Fried's surgery.
In Subsection~\ref{subsec: verify one-sided},
we verify that the taut foliation $\F_\alpha(\s)$ on $M(\s)$
is either $\R$-covered or has one-sided branching
in the positive direction.

\subsection{Acknowledgements}

The author is grateful to Nathan Dunfield for sharing 
the census of examples \cite{Dun} and related information;
to Danny Calegari for explaining some of the background and 
motivations for Question~\ref{que: R-covered exist?}; 
to Sergio Fenley for helpful conversations and for answering some questions.
The author would like to thank David Gabai and Mehdi Yazdi for 
valuable conversations and discussions related to this work.
The author appreciates Saul Schleimer for helpful advice; 
and Qingfeng Lyu, Chi Cheuk Tsang, and Jonathan Zung for beneficial conversations.
The author would like to thank Steven Boyer and Duncan McCoy for 
their support at 
Centre interuniversitaire de recherches en g\'eom\'etrie et topologie
(CIRGET), 
where this project began.
Part of this work was carried out during the author's Spring 2026 residence
at the Simons Laufer Mathematical Sciences Institute (SLMath), 
during the program
``Topological and Geometric Structures in Low Dimensions'', 
with support from NSF grant DMS-2424139.
The author also thanks the organizers of this program at SLMath for 
providing an excellent environment for communication.

\section{Preliminaries}\label{sec: preliminaries}

In this section,
we introduce the basic notation and conventions
used throughout the paper.

\subsection{Conventions}\label{subsec: convention}
Let $A, B$ be metric spaces.
We denote by $A \setminus \setminus B$ 
the closure of $A - B$ under the path metric.

Let $\varphi:\Sigma\to\Sigma$ be
an orientation-preserving pseudo-Anosov homeomorphism on
a compact orientable surface $\Sigma$ with $\partial \Sigma \ne \emptyset$,
and let $M=\Sigma\times I/\stackrel{\varphi}{\sim}$
be the mapping torus of $\varphi$.

\begin{nota}[Distance]\rm
For two slopes $\alpha, \beta$ represented by
simple closed curves on the same boundary component of $M$,
the \emph{distance} between $\alpha$ and $\beta$
is defined to be their minimal geometric intersection number,
denoted by $\Delta(\alpha, \beta)$.
\end{nota}


Fix an orientation on $M$.
We then specify the orientation conventions canonically determined by
the mapping torus structure.

\begin{conv}[Orientation conventions]\rm
We orient the fibered surface $\Sigma\times\{0\} \subseteq M$ so that
the induced normal vector field agrees with 
the increasing orientation on the second coordinate of
the bundle structure $\Sigma\times I$.
This induces an orientation of $\Sigma$.
\end{conv}

Left and right sides of oriented curves in $\Sigma$
are always taken with respect to the induced orientation on $\Sigma$,
and boundary components inherit the orientation
determined by the right-hand rule for properly embedded arcs.

To determine a canonical meridian-longitude coordinate on $T$,
we fix a degeneracy slope on $T$,
as given in Definition~\ref{def: degeneracy locus}.

\begin{conv}[Choice of meridian and longitude]\rm\label{slope}
Let $T$ be a boundary component of $M$,
and let $\delta_T$ denote the degeneracy slope of $T$.
We define the \emph{longitude} of $T$,
denoted by $\lambda_T$,
to be the slope represented by
$T \cap (\Sigma \times \{0\})$ on $T$.
Next, we choose a slope $\mu_T$ on $T$ satisfying
$\Delta(\lambda_T, \mu_T) = 1$ and
\[\Delta(\mu_T,\delta_T) \leqslant \Delta(s,\delta_T)
\quad
\text{for any slope $s$ of $T$ with }
\Delta(\lambda_T,s)=1.\]
We call $\mu_T$ the \emph{meridian} of $T$.
If $\Delta(\lambda_T,\delta_T)\ne 2$,
then $\mu_T$ is uniquely determined by the above conditions.
When $\Delta(\lambda_T,\delta_T)=2$,
there are two possible choices;
we will choose a specific one after 
the slope coordinate system is established.
\end{conv}

For each boundary torus $T\subseteq\partial M$,
the longitude $\lambda_T$ is oriented consistently with
$T \cap (\partial\Sigma\times\{0\})$.
We assign the meridian $\mu_T$ on $T$ an orientation so that
$\mu_T$ can be isotoped to a curve transverse to the fibers $\Sigma\times\{t\}$
and consistent with the increasing orientation on 
the second coordinate of the bundle $\Sigma \times I$.
We normalize algebraic intersection numbers $\langle \cdot,\cdot\rangle$ so that
\[\langle \mu_T,\lambda_T\rangle =
- \langle \lambda_T,\mu_T\rangle = 1.\]
Under this choice,
$(\mu_T,\lambda_T)$ forms a preferred basis for slopes on $T$.
Any essential simple closed curve $\gamma$ on $T$
is identified with the slope
\[\frac{\langle \gamma,\lambda_T\rangle}
     {\langle \mu_T,\gamma\rangle}
\in \mathbb{Q}\cup\{\infty\}.\]

Now suppose that $\Delta(\lambda_T,\delta_T)=2$.
With this choice,
there are two candidates for $\mu_T$,
denoted by $\mu'_T$ and $\mu''_T$,
such that $\delta_T$ has slope $-2$ and $2$
in the canonical coordinate systems given by
$(\mu'_T,\lambda_T)$ and $(\mu''_T,\lambda_T)$, respectively.
We define $\mu_T = \mu''_T$. 
Then $\delta_T$ has slope $2$.

The above conventions remain fixed
throughout the remainder of the paper.

\begin{rmk}\rm
When $M$ is the exterior of a knot in $S^3$,
it follows from \cite[Corollary~7.4]{R01} that
this canonical choice of meridian and longitude
coincides with the standard meridian-longitude convention
for knots in $S^3$ whenever
$\Delta(\lambda_T,\delta_T) \ne 2$.
Otherwise,
the standard meridian is one of $\mu'_T$ and $\mu''_T$.
See \cite[Remark~2.4]{Z26} for some explanation.
\end{rmk}

\subsection{Branching behaviors of taut foliations}\label{subsec: branching behavior}

Let $\F$ be a taut foliation of a closed orientable $3$-manifold $M$,
and let $\w\F$ denote its pullback to the universal cover $\w M$.
We write $L(\F)$ for the leaf space of $\w\F$.
The deck action of $\pi_1(M)$ on $\w M$ naturally induces
an action on $L(\F)$,
which we refer to as the \emph{$\pi_1$-action on the leaf space}.

The leaf space $L(\F)$ is an orientable, connected, and simply connected $1$-manifold.
It is second countable, but need not be Hausdorff;
indeed, if $L(\F)$ is Hausdorff, then it is homeomorphic to $\R$.
We refer the reader to \cite{HaefR57, Pal78} 
for the structure of leaf spaces of foliations.

Suppose that $\F$ is co-orientable.
Then any choice of co-orientation on $\F$ induces 
a canonical orientation on the leaf space $L(\F)$.
In this case, $\pi_1(M)$ acts on $L(\F)$ by orientation-preserving homeomorphisms.
Henceforth, once a co-orientation on $\F$ is fixed,
we always assume that $L(\F)$ is oriented accordingly.

Here we introduce some terminology for describing the branching behavior of
the leaf space $L(\F)$,
following \cite[Chapter 4.7, p. 169]{Cal07}.
Two points $x,y \in L(\F)$ are said to be \emph{comparable}
if either $x = y$ or if there exists an embedded path in $L(\F)$ between $x$ and $y$;
otherwise, they are said to be \emph{incomparable}.
Once an orientation on $L(\F)$ is fixed,
the leaf space $L(\F)$ admits a strict partial order ``$\stackrel{_{L(\F)}}{>}$''
such that, for any two distinct points $x,y \in L(\F)$ which are comparable,
we write $x \stackrel{_{L(\F)}}{>} y$ if
there exists a positively oriented embedded path in $L(\F)$ from $y$ to $x$,
and write $x \stackrel{_{L(\F)}}{<} y$ otherwise.

\begin{defn}[$\R$-covered taut foliations]\rm\label{R-covered}
A taut foliation $\F$ is said to be \emph{$\R$-covered} if
its leaf space $L(\F)$ is homeomorphic to $\R$.
\end{defn}

An $\R$-covered foliation can be co-orientable or non-co-orientable,
whereas a taut foliation with one-sided branching is necessarily co-orientable
(see, for example, \cite[Subsection~2.1]{Cal03}).
To distinguish the direction of branching,
we fix a co-orientation on the taut foliation from the outset.

\begin{defn}[Taut foliations with one-sided branching]\rm\label{one-sided branching}
A co-oriented taut foliation $\F$ is said to have \emph{one-sided branching} if
its leaf space $L(\F)$ is not homeomorphic to $\R$ and satisfies
one of the following conditions:
\begin{itemize}
\item
For any $x,y \in L(\F)$, there exists $z \in L(\F)$ with
$z \stackrel{_{L(\F)}}{>} x,y$.
In this case, $\F$ has \emph{branching in the negative direction}.

\item
For any $x,y \in L(\F)$, there exists $z \in L(\F)$ with
$z \stackrel{_{L(\F)}}{<} x,y$.
In this case, $\F$ has \emph{branching in the positive direction}.
\end{itemize}
\end{defn}

\begin{defn}[Taut foliations with two-sided branching]\rm\label{two-sided branching}
A taut foliation $\F$ is said to have \emph{two-sided branching} if
$\F$ is non-$\R$-covered and does not have one-sided branching.
\end{defn}

Branching behavior is characterized in terms of the
\emph{positive} and \emph{negative ends} of $L(\F)$;
see \cite[Subsection 2.2]{Z25} for a description of ends of $L(\F)$
and \cite[Corollary 2.8]{Z25} for the following characterization.

\begin{prop}
\textnormal{(a)}
A taut foliation $\F$ has one-sided branching if and only if,
up to a choice of orientation on $L(\F)$,
the leaf space $L(\F)$ has either
a unique positive end and infinitely many negative ends,
or a unique negative end and infinitely many positive ends.

\textnormal{(b)}
A taut foliation $\F$ has two-sided branching if and only if,
up to a choice of orientation on $L(\F)$,
there are infinitely many positive ends and infinitely many negative ends in $L(\F)$.
\end{prop}

A co-orientable $\R$-covered foliation implies that the ambient $3$-manifold
has left-orderable fundamental group;
see, for example, \cite[Proposition~5.3]{BRW05} and \cite[Corollary~7.10]{CD03}.
Moreover, it induces a faithful orientation-preserving action
of the fundamental group on $\R$~\cite[Theorem~7.10]{CD03}.
By \cite{Con59} (see also \cite[Theorem 2.4]{BRW05}),
such an action on $\R$ determines
an explicit left-invariant order on the group.

We now turn to the left-orderability coming from 
taut foliations with one-sided branching.
To begin with, we recall the group $\G_\infty$,
a natural quotient of $\Home(\R)$
obtained by identifying homeomorphisms that have the same germ at $+\infty$.

\begin{defn}\rm\label{G-infty}
Let $\sim$ be the equivalence relation on $\Home(\R)$ defined by
$f \sim g$ if $f|_{[n,+\infty)} = g|_{[n,+\infty)}$ for some $n \in \R$.
Define
\[\G_\infty = \Home(\R) / \sim.\]
Composition in $\Home(\R)$ induces a well-defined group operation on $\G_\infty$.
We call $\G_\infty$ the \emph{group of germs at $\infty$}.
\end{defn}

The following theorem is due to Navas; 
see \cite[Proposition~2.2]{M15} for a proof.

\begin{THM}[Navas]\label{Navas}
    The group $\G_\infty$ is left-orderable.
\end{THM}

We note from \cite{M15} that,
although the group $\G_\infty$ is left-orderable,
it admits no nontrivial action on $\R$.

The left-orderability for co-orientable taut foliations with one-sided branching was
proved by the author in~\cite{Z25}.

\begin{THM}[{\cite{Z25}}]\label{LO one-sided branching}
    Let $M$ be a closed orientable $3$-manifold that admits 
    a co-orientable taut foliation $\F$ with one-sided branching.
    Then $\pi_1(M)$ is left-orderable.
    In addition,
    there exists a faithful representation
    \[d: \pi_1(M) \to \G_\infty.\]
\end{THM}

The proof of Theorem \ref{LO one-sided branching} in \cite{Z25} only uses
the property that $\F$ has no branching in one of the two directions,
so the same argument also applies when $\F$ is $\R$-covered.
Thus,
if a closed $3$-manifold admits a co-orientable taut foliation which 
either has one-sided branching or is $\R$-covered,
then this foliation induces a faithful representation of
its fundamental group into $\G_\infty$.

Finally, we describe known sources of taut foliations with two-sided branching.
Recall that \emph{genuine} essential laminations were introduced
~\cite{GK98a, GK98b} as essential laminations that cannot be obtained
by splitting open taut foliations along leaves.

\begin{defn}\rm
An essential lamination is \emph{genuine} if
it has at least one complementary region
that is not homeomorphic to a surface bundle over a closed interval.
\end{defn}

We record the following consequence,
which implies that many taut foliations have two-sided branching.

\begin{prop}\label{genuine two-sided branching}
Let $\Ll$ be an essential sublamination of a taut foliation $\F$.
If $\Ll$ is genuine, then $\F$ has two-sided branching.
\end{prop}

\begin{proof}
We argue by excluding the cases where $\F$ is $\R$-covered or
has one-sided branching.
Recall that a lamination is \emph{minimal} if 
every sublamination is equal to the lamination itself.

First suppose that $\F$ is $\R$-covered. 
Then $\F$ is not minimal since $\Ll \varsubsetneqq \F$;
let $\Ll_0$ be a minimal sublamination of $\Ll$. 
As illustrated in the proof of \cite[Proposition 2.6]{Fen02}, 
every complementary region $C_0$ of $\Ll_0$ 
is homeomorphic to an oriented $I$-bundle over a surface, 
and the restriction of $\F$ to $C_0$ is a \emph{foliated $I$-bundle}, 
meaning that 
$\F$ is transverse to the $I$-fibers of $C_0$ \cite[Definition~2.3]{Fen02}
(we note that \cite[Proposition 2.6]{Fen02} assumes that
$\F$ has no compact leaf,
while this assumption is not needed in this part of the proof).
Since $\Ll_0 \subseteq \Ll$, 
any complementary region $C$ of $\Ll$ is contained in 
some complementary region $C_0$ of $\Ll_0$. 
The $I$-bundle structure on $C_0$ restricts to 
an $I$-bundle structure on $C$, 
so $C$ is homeomorphic to an oriented $I$-bundle over a surface.
It follows that every complementary region of 
$\Ll$ is a surface bundle over an interval. 
Hence $\Ll$ is not genuine, a contradiction.

Next suppose that $\F$ has one-sided branching.
Let $\Ll_0$ be a minimal sublamination of $\Ll$.
By \cite[Theorem~2.2.5]{Cal03},
each complementary region of $\Ll_0$ in $M$
is an \emph{inessential pocket}
in the sense of \cite[Definition~2.2.3]{Cal03}.
As illustrated in the proof of \cite[Theorem~2.2.7]{Cal03},
any inessential pocket is homeomorphic to 
a surface bundle over a closed interval.
More precisely,
there is a taut foliation $\F'$ of $M$ such that
$\F$ is obtained from $\F'$ by blowing-up some leaves
and possibly perturbing each blown-up region to certain foliated $I$-bundle,
with each complementary region of $\Ll_0$ 
a foliated $I$-bundle obtained from perturbing some blown-up region.
As in the proof of the previous case,
every complementary region of $\Ll$ is a surface bundle over an interval,
so $\Ll$ is not genuine, a contradiction.
Therefore, $\F$ contains no genuine essential sublamination in this case.

Thus, $\F$ must have two-sided branching whenever $\Ll$ is genuine.
\end{proof}

\subsection{Pseudo-Anosov flows}\label{subsec: pseudo-Anosov flow}

Pseudo-Anosov flows can be formulated in several ways.
We work with a topological definition; see \cite[Definition~7.1]{Fen02}.

\begin{defn}[Pseudo-Anosov flows]\rm\label{topological pseudo-Anosov flow}
Let $M$ be a closed orientable $3$-manifold.
Fix a Riemannian metric
$d(\cdot,\cdot)$ on $M$.
Let $\phi^{t}\colon M\to M$ $(t\in\R)$ be a continuous flow on $M$.
The flow $\phi$ is called a \emph{pseudo-Anosov flow}
if the following conditions are satisfied.

\begin{enumerate}[(i)]

\item
There exist two (possibly singular) $2$-dimensional foliations 
$\Fws(\phi)$ and $\Fwu(\phi)$ of $M$,
called the \emph{weak stable foliation} and 
\emph{weak unstable foliation} of $\phi$, respectively,
such that every orbit of $\phi$ is contained in a leaf of each foliation.
There is a (possibly empty) finite collection of closed orbits
$\gamma_1,\ldots,\gamma_n$ which form precisely the singular set of $\Fws$ and $\Fwu$.
Away from $\bigcup_{i=1}^n \gamma_i$,
the restrictions of $\Fws(\phi)$ and $\Fwu(\phi)$ 
are regular foliations transverse to each other.
Each $\gamma_i$ is called a \emph{singular orbit} of $\phi$,
and each leaf of $\Fws(\phi)$ or $\Fwu(\phi)$ containing a singular orbit
is homeomorphic to a $p$-prong bundle over $S^1$ for some $p \geqslant 3$.

\item
Suppose that $w,x \in M$ are two points lying in the same leaf of $\Fws(\phi)$,
then there exists $h \in \Home(\R)$ such that
\[\lim_{t\to+\infty} d(\phi^{t}(w),\phi^{h(t)}(x)) = 0.\]
Suppose that $y,z \in M$ are two points lying in the same leaf of $\Fwu(\phi)$,
then there exists $h \in \Home(\R)$ such that
\[\lim_{t\to-\infty} d(\phi^{t}(y),\phi^{h(t)}(z)) = 0.\]

\item
For two points $w,x \in M$ 
lying in two distinct orbits of the same leaf $\lambda$ of $\Fws(\phi)$,
the distance 
$d_\lambda(\phi^t(w),\phi^\R(x))$ is sufficiently large when $t \to -\infty$.
For two points $y,z \in M$ 
lying in two distinct orbits of the same leaf $\mu$ of $\Fwu(\phi)$,
the distance
$d_\mu(\phi^t(y),\phi^\R(z))$ is sufficiently large when $t \to +\infty$.
Here,
$d_\lambda$ and $d_\mu$ denote the path metric on $\lambda$ and $\mu$ induced from
$d(\cdot,\cdot)$, respectively.
\end{enumerate}
\end{defn}

By convention, the orbit $\phi^\R(x)$ of $x \in M$ is denoted by $\phi(x)$.
For each singular leaf $l$ of $\Fws(\phi)$ or $\Fwu(\phi)$ 
with singular orbit $\gamma$,
each component of $l \setminus \setminus \gamma$ is called 
a \emph{half-leaf} of $\phi$.
If $l'$ is a regular leaf of $\Fws(\phi)$ or $\Fwu(\phi)$
containing a closed orbit $\gamma'$,
we also call each component of $l' \setminus \setminus \gamma'$ a half-leaf.

Suspension flows in pseudo-Anosov mapping tori
are natural analogues of pseudo-Anosov flows
for manifolds with toroidal boundary.

\begin{defn}[Suspension flows]\rm\label{def: suspension flow}
Let $\varphi:\Sigma\to\Sigma$ be an orientation-preserving pseudo-Anosov homeomorphism
on a compact orientable surface $\Sigma$ with nonempty boundary.
Let $\Fs$ and $\Fu$ be the stable and unstable foliations of $\varphi$,
and let
$M=\Sigma\times I/\stackrel{\varphi}{\sim}$ be the mapping torus of $\varphi$.
We orient each fiber of the interval bundle
\[\{(x,I)\mid x\in\Sigma\}\subseteq \Sigma\times I\]
by the increasing direction in the second coordinate.
Under the equivalence relation $\stackrel{\varphi}{\sim}$,
these oriented fibers descend to a well-defined flow $\psi$ on $M$,
called the \emph{suspension flow} of $\varphi$.
The foliations $\Fs, \Fu$ on $\Sigma$ suspend to
a pair of singular foliations $\Fws(\psi), \Fwu(\psi)$ on $M$,
referred to as the \emph{weak stable foliation} and
\emph{weak unstable foliation} of $\psi$, respectively.
\end{defn}


For each leaf $l$ of $\Fs$ disjoint from $\partial \Sigma$ that
contains a closed orbit $\gamma$,
each component of $l \setminus \setminus \gamma$
is called a \emph{half-leaf} of $\Fws(\psi)$.
The suspension of each prong of $\Fs$ at a boundary singularity
is also called a half-leaf.
We use the same terminology for $\Fwu(\psi)$.

We now describe the structure of $\Fs$ and $\Fu$ near $\partial \Sigma$.
Let $C$ be a component of $\partial \Sigma$,
and let $u_1, \ldots, u_k$ be the singularities of $\Fs$ on $C$.
Each $u_i$ has exactly three separatrices:
two are contained in $C$, and the third is
a prong pointing into the interior of $\Sigma$.
Under the homeomorphism $\varphi$,
there are exactly $2k$ periodic points on $C$:
$k$ of them are \emph{attracting periodic points},
namely $u_1,\ldots,u_k$,
and the remaining $k$ are \emph{repelling periodic points},
which are precisely the singularities of $\Fu$ on $C$.
See \cite[Appendix]{CanC99} and \cite[p. 466]{R01} for detailed descriptions.



As introduced in \cite[Section~5]{GO89},
periodic orbits of $\psi$ give rise to
the following canonical local invariant associated to each boundary component of $M$.
See also \cite[Section~8]{Gab97} and \cite{R01}.

\begin{defn}[Degeneracy locus]\rm\label{def: degeneracy locus}
Under the assumptions of Definition \ref{def: suspension flow},
let $T_1,\ldots,T_r$ be the boundary components of $M$.
The set of periodic orbits on $T_i$
consists of $2n_i$ parallel essential simple closed curves for some $n_i \in \N_+$:
$n_i$ of them arise from attracting periodic points of $\varphi$,
and the remaining $n_i$ arise from repelling periodic points of $\varphi$.
We call the common slope of these curves the \emph{degeneracy slope} of $T_i$,
denoted by $\delta_{T_i}$.
The \emph{degeneracy locus} of $T_i$ is defined to be
the union of those $n_i$ periodic orbits arising from attracting periodic points.
We identify it with
\[d(T_i)=n_i \cdot \delta_{T_i},\]
where $n_i$ is called the \emph{multiplicity} of $d(T_i)$.
\end{defn}

By convention, we set
\[\Delta(\alpha,d(T_i)) = n_i \cdot \Delta(\alpha,\delta_{T_i})\]
for any rational slope $\alpha$ on $T_i$.

Fried \cite{Fri83} introduced an operation
producing pseudo-Anosov flows from the suspension flow
on pseudo-Anosov mapping tori of closed surfaces.
To adapt this construction to our setting,
we describe an equivalent formulation for $(M,\psi)$
in terms of Dehn fillings along $\partial M$.

\begin{cons}[{\cite{Fri83}}]\rm\label{Fried}\rm
Under the assumptions in Definition \ref{def: degeneracy locus},
for each $1 \leqslant i \leqslant r$ choose a slope $s_i$ on $T_i$
with $\Delta(s_i,d(T_i)) \geqslant 2$.
Perform Dehn filling along each $T_i$ with slope $s_i$,
and let $N$ denote the resulting manifold.
Let $V_i$ be the filling solid torus attached to $T_i$.
By collapsing each $V_i$ to its core curve $\rho_i$
and identifying the complement
$N-\bigcup_{i=1}^r \rho_i$
with $\Int(M)$,
the restriction $\psi\mid_{\Int(M)}$ extends uniquely to a flow $\psi'$ on $N$,
which is a pseudo-Anosov flow.
The weak stable and unstable foliations
$\Fws(\psi), \Fwu(\psi)$
induce a pair of singular foliations on $N$,
which are precisely the stable and unstable foliations of $\psi'$.
Each $\rho_i$ has $\Delta(s_i,d(T_i))$ prongs
in the leaves of these foliations that contain it.
\end{cons}

Compare this construction with Goodman's surgery~\cite{Goo83}.
An equivalence between these two types of surgeries
is discussed in \cite[Chapter 0, Section 2, pp.~xix--xx]{Sha20}.

From the perspective of the universal cover,
the structure of a pseudo-Anosov flow can be described by its orbit space,
which we now define.

\begin{defn}\rm\label{orbit space}
    Let $\phi$ be a pseudo-Anosov flow in a closed $3$-manifold $N$,
    and let $p: \w N \to N$ denote the universal cover of $N$. 
    Then $\phi$ lifts to a flow $\w \phi$ in $\w N$,
    and the weak stable and unstable foliations $\Fws(\phi), \Fwu(\phi)$ lift to
a pair of (possibly singular) foliations $\w{\Fws(\phi)}, \w{\Fwu(\phi)}$ of $\w N$.
By projecting each flowline of $\w \phi$ to a single point,
we obtain a quotient space $\Or(\phi)$ of $\w N$ endowed with
the quotient topology from $\w \phi$.
The space $\Or(\phi)$ is referred to as the \emph{orbit space} of $\phi$.
The foliations $\w{\Fws(\phi)}, \w{\Fwu(\phi)}$ descend to
a pair of one-dimensional singular foliations,
denoted $\Ors(\phi), \Oru(\phi)$ respectively.
\end{defn}

It was proved in \cite[Proposition 4.2]{FenM01} that

\begin{THM}
The orbit space $\Or(\phi)$ of any pseudo-Anosov flow $\phi$
is homeomorphic to a topological $2$-plane.
\end{THM}

Under Definition~\ref{orbit space},
each component of the preimage of a regular leaf of 
$\Fws(\phi)$ or $\Fwu(\phi)$
is called a regular leaf in 
$\w{\Fws(\phi)}$ or $\w{\Fwu(\phi)}$,
and each component of the preimage of a half-leaf of 
$\Fws(\phi)$ or $\Fwu(\phi)$
is called a half-leaf in 
$\w{\Fws(\phi)}$ or $\w{\Fwu(\phi)}$.
For any $x \in \w N$, denote the orbit of $\w \phi$ through $x$ by $\w \phi(x)$,
parametrized by
\[\w \phi(x) = \{\w \phi^t(x) \mid t \in \R\}\]
so that $p(\w \phi^t(x)) = \phi^t(p(x))$.
We use the same terminology for the pullback of a suspension flow $\psi$.

\subsection{The branched surface}\label{subsec: branched surface}

Branched surfaces provide a useful tool for constructing
foliations and laminations in $3$-manifolds.
In this subsection, we review the basic notions needed in this paper.

\begin{figure}
	\centering
	\subfigure[]{
	\includegraphics[width=0.25\textwidth]{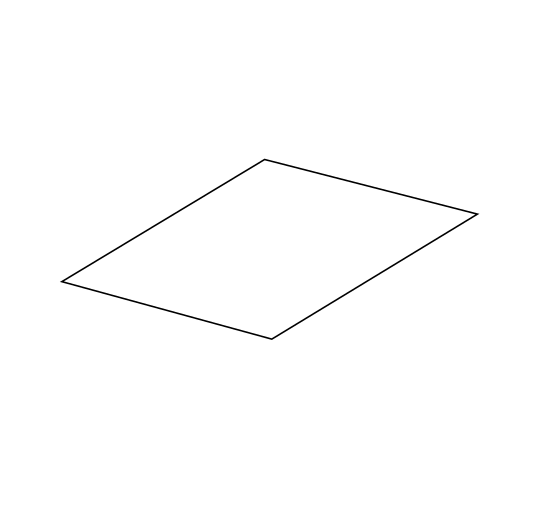}}
\subfigure[]{
	\includegraphics[width=0.25\textwidth]{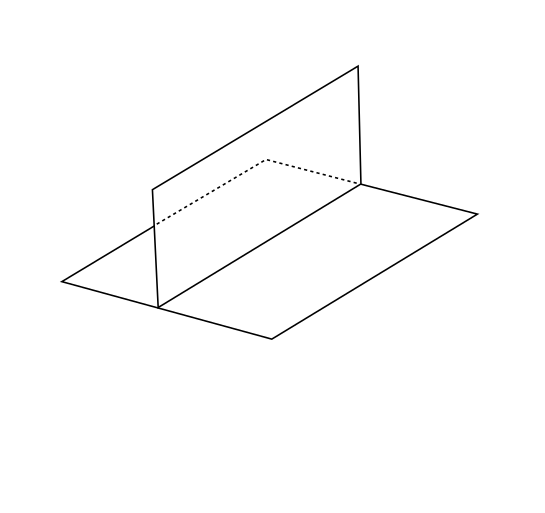}}
	\subfigure[]{
	\includegraphics[width=0.25\textwidth]{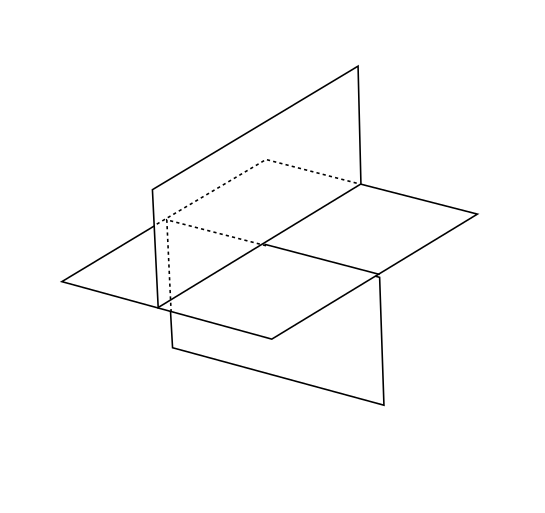}}
	\subfigure[]{
		\includegraphics[width=0.25\textwidth]{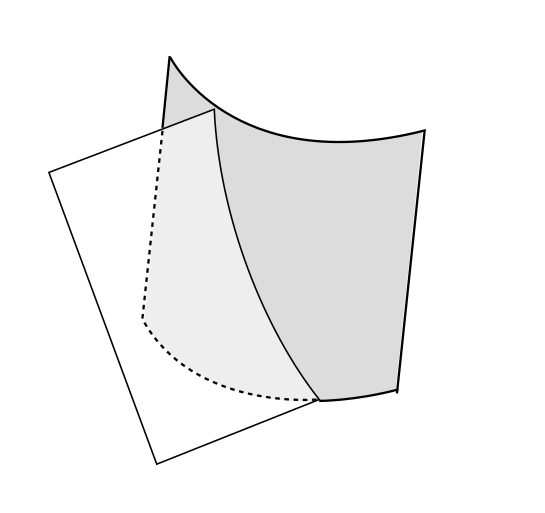}}
	\subfigure[]{
		\includegraphics[width=0.25\textwidth]{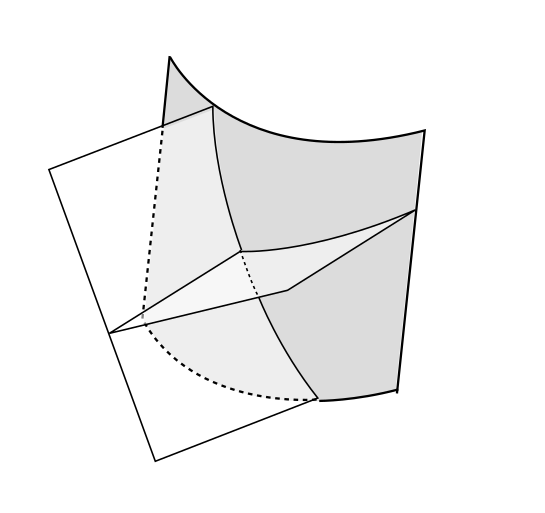}}
	\caption{Local models of a standard spine in a $3$-manifold $M$ of various types.
In (d) and (e), the standard spine has boundary on $\partial M$,
and the shaded regions lie in $\partial M$.
}\label{standard spine picture}
\end{figure}

\begin{defn}\rm \label{standard spine}
Let $M$ be a compact orientable $3$-manifold.
An embedded $2$-complex $X$ in $M$ is called a
\emph{standard spine} if each point of $X$ admits a local model
of the types shown in Figure~\ref{standard spine picture}.
More precisely,
points in $X - \partial M$ have neighborhoods modeled on
Figure~\ref{standard spine picture} (a)$\sim$(c),
and points in $X \cap \partial M$ have neighborhoods modeled on
Figure~\ref{standard spine picture} (d) or (e).
\end{defn}

\begin{defn}\rm
Let $M$ be a compact orientable $3$-manifold. 
A \emph{branched surface} $B$ in $M$ is 
a subspace homeomorphic to a standard spine that 
has a well-defined tangent plane at each point. 
Each point of $B$ has a neighborhood modeled as in Figure~\ref{fig: branched surface}.
\end{defn}

For a branched surface $B$ in a compact orientable $3$-manifold $M$,
a \emph{fibered neighborhood} $N(B)$ of $B$ is 
a neighborhood of $B$ in $M$ foliated by closed intervals,
associated with a quotient map $\pi \colon N(B) \to B$ that sends
each closed interval to a single point of $B$;
see Figure~\ref{fig: fibered neighborhood} (b) for a local model of
the bundle structure of these closed intervals.
We refer to each closed interval as an \emph{interval fiber} of $N(B)$,
and to the map $\pi$ as the \emph{collapsing map}.

\begin{figure}
	\centering
	\subfigure[]{
		\includegraphics[width=0.25\textwidth]{19.png}}
	\subfigure[]{
		\includegraphics[width=0.25\textwidth]{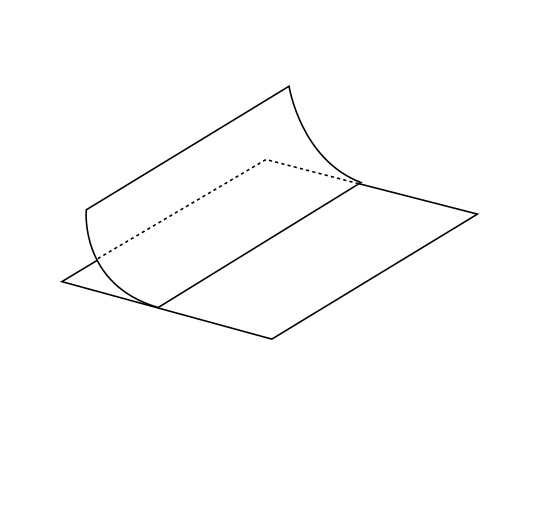}}
	\subfigure[]{
		\includegraphics[width=0.25\textwidth]{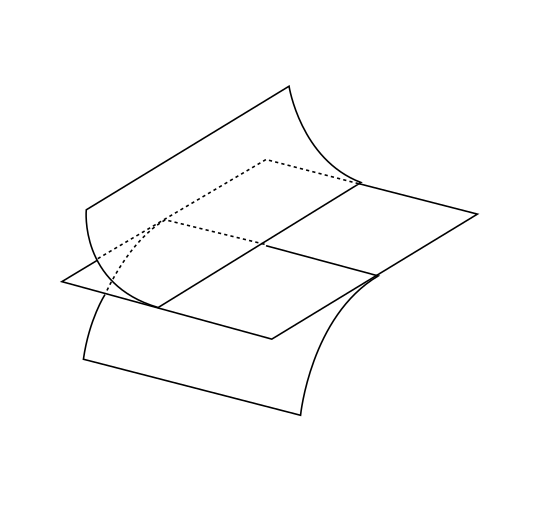}}
	\subfigure[]{
		\includegraphics[width=0.25\textwidth]{24.png}}
	\subfigure[]{
		\includegraphics[width=0.25\textwidth]{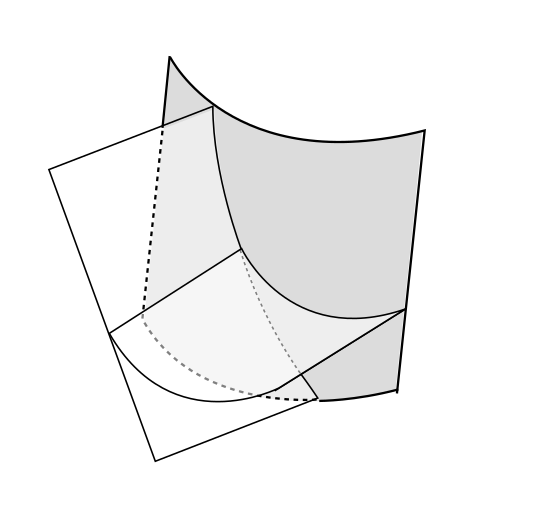}}
	\caption{Local models of a branched surface in 
    a $3$-manifold $M$ of various types.
In (d) and (e), the branched surface has boundary on $\partial M$,
which is shown shaded.}\label{fig: branched surface}
\end{figure}

The surface $\partial N(B) \setminus \setminus \partial M$
can be considered as the union of 
two (possibly disconnected) compact subsurfaces as follows.
Define $\partial_v N(B)$ to be the union of points in
$\partial N(B) \setminus \setminus \partial M$
that lie in the interior of some interval fiber,
and define $\partial_h N(B)$ to be the closure of the union of points in
$\partial N(B) \setminus \setminus \partial M$
that are endpoints of interval fibers.
Then $\partial_v N(B)$ is tangent to the interval fibers of $N(B)$,
and $\partial_h N(B)$ is transverse to the interval fibers of $N(B)$.
The sets $\partial_v N(B)$ and $\partial_h N(B)$ are
called the \emph{vertical boundary} and the \emph{horizontal boundary}
of $N(B)$, respectively.
See Figure \ref{fig: fibered neighborhood} (b) for a local model.

\begin{defn}[cusp directions]\rm
The \emph{branch locus} of $B$, denoted $L(B)$,
is the subset of $B$
consisting of all points that have no Euclidean neighborhood in $B$.
Then $L(B)$ is a graph.
For every edge $e$ of $L(B)$,
there is a component $V$ of $\partial_v N(B)$ for which
$e \subseteq \pi(V)$.
We associate the edge $e$ with a normal vector lying in $B$,
induced from the normal vector on $V$ pointing inward to $N(B)$.
The normal direction at $e$ in $B$ represented by this normal vector 
is called the \emph{cusp direction} at $e$.
\end{defn}

\begin{defn}\rm
Let $B$ be a branched surface in $M$.
A lamination $\Ll$ of $M$ is said to be \emph{carried} by $B$ if,
for some fibered neighborhood $N(B)$ of $B$,
$\Ll$ is contained in $N(B)$ and is transverse to the interval fibers of $N(B)$.
Under this assumption,
the lamination $\Ll$ is said to be \emph{fully carried} by $B$ if
every interval fiber of $N(B)$ intersects some leaf of $\Ll$.
\end{defn}

Each component of $B \setminus \setminus L(B)$ is called
a \emph{branch sector} of $B$.
A branched surface is said to be \emph{co-orientable} if
it admits a continuously varying normal vector field.
If a lamination $\Ll$ is carried by a co-orientable branched surface $B$,
then any co-orientation on $B$ induces a co-orientation on $\Ll$.

\begin{figure}
\centering
	\subfigure[]{
		\includegraphics[width=0.45\textwidth]{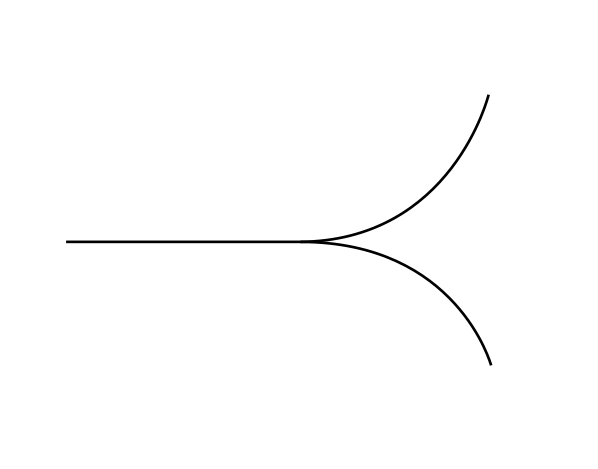}}
	\subfigure[]{
		\includegraphics[width=0.45\textwidth]{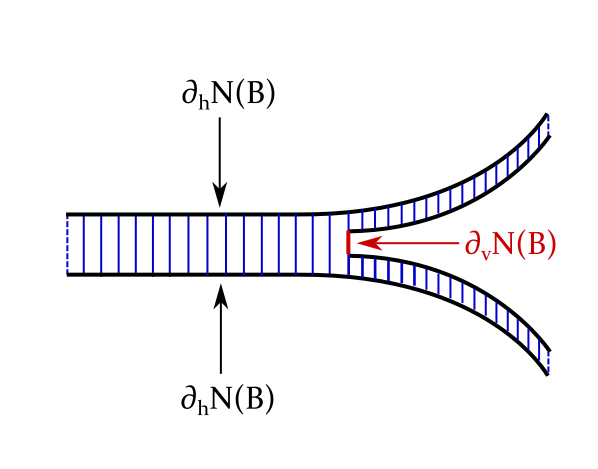}}
	\caption{(a) A local model of a branched surface $B$.
(b) The corresponding local model of 
a fibered neighborhood $N(B)$ of $B$.}\label{fig: fibered neighborhood}
\end{figure}


We refer the reader to \cite{GO89} for the theory of
essential laminations and essential branched surfaces,
and to \cite{Li02, Li03} for the theory of laminar branched surfaces.
These notions provide the background for the branched surfaces
used in the construction of Theorem~\ref{foliation}.
We turn to these specific branched surfaces in the next section.

\section{Foliations associated to admissible systems of arcs}\label{sec: branched surface}

The purpose of this section is to set up a co-orientable taut foliation
associated to each admissible system of arcs
on every Dehn filling considered in Theorem~\ref{main}.
We work under the assumptions of Theorem~\ref{main}.
Recall that $\Sigma$ is a compact orientable surface,
$\varphi$ is an orientation-preserving pseudo-Anosov homeomorphism of $\Sigma$
with invariant foliations $\Fs$ and $\Fu$,
and $M$ is the mapping torus $\Sigma \times I / \stackrel{\varphi}{\sim}$
with boundary components $T_1,\ldots,T_r$.
In addition,
$\Fs$ is co-orientable and $\varphi$ reverses its co-orientation.


\subsection{Branched surfaces associated to admissible systems of arcs}\label{subsec: admissible system of arcs}

In this subsection, we review the construction from~\cite{Z26}
of the branched surfaces used in Theorem~\ref{foliation}.
In~\cite[Subsection~3.1]{Z26},
a branched surface is constructed from 
a finite system of properly embedded arcs chosen in~\cite[Construction~3.2]{Z26}.
For our purposes, we will not need to use 
the particular system of arcs used there. 
Instead, we extract the properties of that system of arcs which
ensure that the resulting branched surface
fully carries essential laminations extending to co-orientable taut
foliations on the Dehn fillings considered in Theorem~\ref{foliation}.
We therefore make the following definition.

\begin{defn}\rm\label{system of arcs}
A family of oriented properly embedded arcs $\alpha = (\alpha_1,\ldots,\alpha_n)$
is called an \emph{admissible system of arcs} 
with respect to $(\Sigma,\varphi)$ if the following conditions hold.
\begin{enumerate}[(i)]
    \item Each $\alpha_i$ ($1 \leqslant i \leqslant n$)
    is positively transverse to $\Fs$ and
    is disjoint from the singularities of $\Fs$.

    \item The set of endpoints $\bigcup_{i=1}^{n} \partial \alpha_i$
    is disjoint from its image under $\varphi$.

    \item For each open segment of
    \[\partial \Sigma - \{\text{singularities of } \Fs\},\]
    there is a unique element of $\bigcup_{i=1}^{n} \partial \alpha_i$
    lying on that segment.
\end{enumerate}
\end{defn}

Note that condition~(i) implies that each endpoint of $\alpha_i$
is not a boundary singularity of $\Fs$.
By condition~(iii), $n$ equals half the total number of
boundary singularities of $\Fs$.

Given an admissible system of arcs,
a branched surface is constructed as in
\cite[Definition~3.4]{Z26}.

\begin{defn}\rm\label{branched surface}
Let $\alpha = (\alpha_1,\ldots,\alpha_n)$ be an admissible system of arcs 
with respect to $(\Sigma,\varphi)$.
We construct a branched surface $B'_\alpha$ in $M$ as follows.
\begin{enumerate}
    \item
    Let
    \[S=(\Sigma \times \{0,1\}) \cup \bigcup_{i=1}^{n}(\alpha_i \times I)
    \subseteq \Sigma \times I\]
    be a standard spine in the product space $\Sigma \times I$.
    For each $1 \leqslant i \leqslant n$,
    the product $\alpha_i \times I$
    is called a \emph{product disk},
    and its two intersection arcs with $\Sigma \times \{0\}$ and $\Sigma \times \{1\}$
    are assigned cusp directions as follows.
    The arc $\alpha_i \times \{0\}$
    has the cusp direction pointing to its left side in $\Sigma \times \{0\}$,
    while the arc $\alpha_i \times \{1\}$
    has the cusp direction pointing to its right side in $\Sigma \times \{1\}$.

    \item
    Let
    \[q:\Sigma \times I \to M\]
    be the canonical quotient map identifying $(x,1)$ with $(\varphi(x),0)$ 
    for all $x \in \Sigma$.
    Let $B'_\alpha$ be the image of $S$ under $q$.
    For each product disk $\alpha_i \times I$,
    the two arcs $\alpha_i \times \{0\}$ and $\alpha_i \times \{1\}$
    are both mapped into the fibered surface $\Sigma \times \{0\}$ of $M$.
    We call the image of $\alpha_i \times \{0\}$ under $q$ the \emph{lower arc}
    of $\alpha_i \times I$,
    and the image of $\alpha_i \times \{1\}$ the \emph{upper arc} 
    of $\alpha_i \times I$.

    \item
Let $\beta_i$ ($1 \leqslant i \leqslant n$) be 
an oriented properly embedded arc on $\Sigma$
isotopic to $\varphi(\alpha_i)$ relative to its endpoints such that
each $\beta_i$ is transverse to $\Fs$,
the arcs $\beta_1,\ldots,\beta_n$ are pairwise disjoint, and
they have only double intersection points with 
$\bigcup_{i=1}^{n}\alpha_i$.
We isotope the union of product disks $\bigcup_{i=1}^{n}(\alpha_i \times I)$
    relative to
    \[\bigcup_{i=1}^{n}(\alpha_i \times \{0\}) \cup
      \bigcup_{i=1}^{n}(\partial \alpha_i \times I)\]
    so that the upper arc of each product disk $\alpha_i \times I$
    is isotoped to $\beta_i \times \{0\}$.
\end{enumerate}
\end{defn}

We note that each arc $\beta_i$ is negatively transverse to $\Fs$,
since $\varphi$ is co-orientation-reversing.

Using the theory of laminar branched surfaces~\cite{Li02, Li03},
it was shown in~\cite[Subsection~3.2]{Z26} that
the branched surface $B'_\alpha$ fully carries
essential laminations, and in~\cite[Subsections~3.3, 3.4]{Z26} that
such laminations can be chosen to intersect
each boundary component $T_i$ ($1 \leqslant i \leqslant r$) in
simple closed curves with any slope in
$J_i \cap (\Q \cup \{\infty\})$,
where $J_i$ is the interval of slopes defined in Theorem \ref{foliation}.
Therefore, such laminations extend to co-orientable taut foliations in
the corresponding Dehn fillings
(see~\cite[Subsection~3.5]{Z26}).
We summarize this as follows.

\begin{THM}[\cite{Z26}]\label{thm: construction from branched surface}
Let $s_i \in J_i \cap (\Q \cup \{\infty\})$ 
for each $1 \leqslant i \leqslant r$,
and let $\s$ denote the multislope $(s_1,\ldots,s_r)$.
Then the branched surface $B'_\alpha$ fully carries
an essential lamination $\Ll'_\alpha$ which intersects
each $T_i$ in a union of parallel simple closed curves of slope $s_i$,
and $\Ll'_\alpha$ extends to a foliation $\F'_\alpha$ with
$\F'_\alpha \mid_{T_i}$ a foliation by simple closed curves of slope $s_i$.
In particular,
the foliation $\F'_\alpha$ extends to 
a co-orientable taut foliation $\F'_\alpha(\s)$ of $M(\s)$.
\end{THM}

\begin{rmk}\label{rmk: independent}
Although the construction in \cite{Z26} is carried out using the
particular system of arcs chosen in~\cite[Construction 3.2]{Z26},
the arguments there only use the properties listed in
Definition \ref{system of arcs}. 
Thus the conclusions summarized in
Theorem \ref{thm: construction from branched surface} hold for every
admissible system of arcs $\alpha$ with respect to $(\Sigma,\varphi)$.
\end{rmk}

\subsection{The modified branched surface $B_\alpha$}\label{subsec: perturb}

The foliations $\F'_\alpha(\s)$ of $M(\s)$ constructed in the previous subsection
are sufficient for the existence statement of taut foliations 
in Theorem~\ref{foliation}.
For the foliations $\F_\alpha(\s)$ used in Theorem~\ref{main}, 
however, we will use a modified version of this construction.
We now describe this modification.



We continue with the setting of the previous subsection.
Let $\psi$ denote the suspension flow of $M$.
Below, we modify the branched surface $B'_\alpha$ constructed previously to 
obtain a new branched surface $B_\alpha$,
which will be transverse to $\psi$ and
fully carry laminations extending to the desired foliations $\F_\alpha(\s)$.

\begin{cons}\rm\label{new branched surface}
Let $\alpha = (\alpha_1,\ldots,\alpha_n)$ be an admissible system of arcs 
with respect to $(\Sigma,\varphi)$,
and let $B'_\alpha$ be a branched surface 
as given in Definition \ref{branched surface}.
We modify the branched surface $B'_\alpha$ as follows.

\begin{enumerate}
    \item
For each $\alpha_i$, 
choose a sufficiently small neighborhood $U_i$ of $\alpha_i$ in $\Sigma$, 
so that $U_1,\ldots,U_n$ are pairwise disjoint and 
each $U_i$ is disjoint from the singularities of $\Fs$ and $\Fu$
and from $\bigcup_{j=1}^n \varphi^{-1}(\partial \alpha_j)$.
Let $\alpha^+_i$ be an oriented properly embedded arc in $\Sigma$ such that
\begin{enumerate}[label=(\roman*), leftmargin=*, align=left]
    \item $\alpha^+_i$ is contained in the component of $U_i - \alpha_i$
    on the right side of $\alpha_i$,
    \item $\alpha^+_i$ is isotopic to $\alpha_i$,
    \item $\alpha^+_i$ is positively transverse to $\Fs$.
\end{enumerate}
Then $\alpha^+_1,\ldots,\alpha^+_n$ are still pairwise disjoint.
We further isotope each $\alpha^+_i$ slightly so that
each $\alpha_i$ intersects $\bigcup_{k=1}^{n} \varphi(\alpha^+_k)$ 
only in transverse double points,
with all the above conditions preserved.

\item
Let $D(\alpha_i)$ denote the product disk of $B'_\alpha$
produced from $\alpha_i \times I$.
We isotope $D(\alpha_i)$, fixing its lower arc, 
so that its upper arc is moved to $\alpha^+_i \times \{1\}$;
under the quotient map $q:\Sigma \times I \to M$, 
this upper arc is identified with
$\varphi(\alpha^+_i) \times \{0\}
\subseteq \Sigma \times \{0\}$.
\end{enumerate}
\end{cons}

We note that condition (i) of Definition \ref{system of arcs}
ensures that each endpoint of $\alpha_i$ lies in 
a neighborhood of $\partial \Sigma$
disjoint from the boundary singularities of $\Fs$,
which guarantees that $\alpha_i$ can be isotoped to 
a parallel arc $\alpha^+_i$ on its right side
while preserving transversality to $\Fs$.

\begin{figure}
	\centering
	\subfigure[]{
	\includegraphics[width=0.35\textwidth]{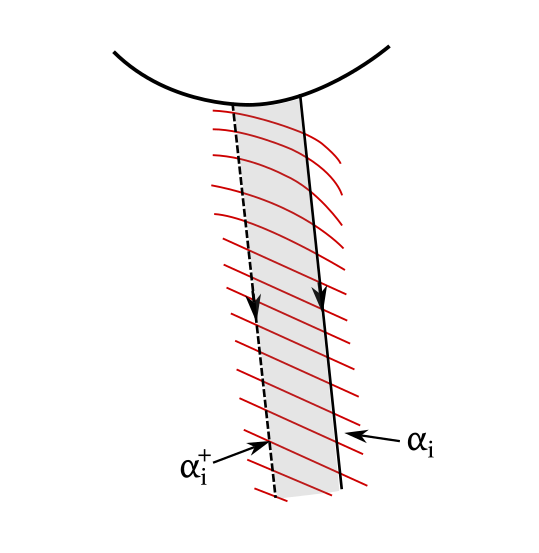}}
\subfigure[]{
	\includegraphics[width=0.35\textwidth]{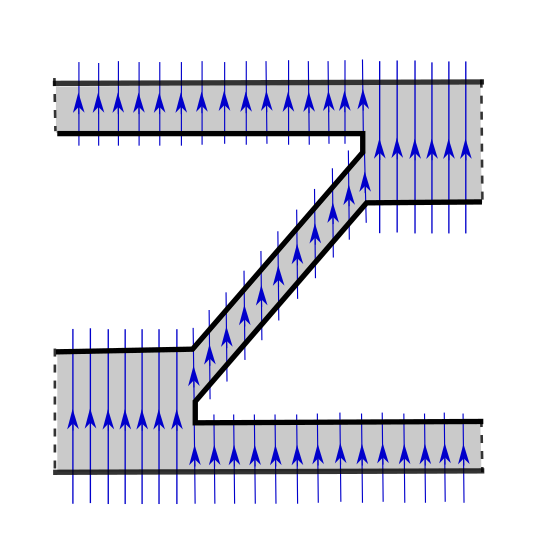}}
	\caption{(a) A local picture of the arc $\alpha^+_i$ 
obtained from $\alpha_i$ by isotopy,
and the band $B_i$ represented by the shaded region.
The red segments represent leaves of $\Fs$.
(b) A local model of $N(B_\alpha)$, 
whose interval fibers are contained in orbits of $\psi$.}\label{fig: modify}
\end{figure}

There exists a properly embedded band $B_i: I \times I \to \Sigma$ with
two horizontal sides in $\partial \Sigma$ and
two vertical sides equal to $\alpha_i$ and $\alpha^+_i$,
respectively,
where the inward normal vectors along $\alpha_i$ and $\alpha^+_i$
point to the right of $\alpha_i$ and to the left of $\alpha^+_i$.
We note that
every leaf of $\Fs$ intersects $B_i$ in a closed segment.
See Figure~\ref{fig: modify} (a) for a local picture of $\alpha^+_i$ and $B_i$.

Since $\alpha_i$ and $\alpha^+_i$ bound the band $B_i$,
we may isotope $D(\alpha_i)$, relative to its lower and upper arcs,
to a section of the subbundle $B_i \times I$ of 
the $I$-bundle structure $\Sigma \times I$ in $M$.
In particular, the projection $\Sigma \times I \to \Sigma$
takes $D(\alpha_i)$ homeomorphically onto $B_i$,
so $D(\alpha_i)$ is transverse to the second coordinate in $B_i \times I$,
and hence transverse to $\psi$.
Since $\Sigma \times \{0\}$ is transverse to $\psi$ and
every product disk is now also transverse to $\psi$,
there exists a fibered neighborhood $N(B_\alpha)$ such that
each interval fiber lies in an orbit of $\psi$;
see Figure~\ref{fig: modify} (b).
Henceforth,
by a fibered neighborhood of $B_\alpha$,
we will always take it to be $N(B_\alpha)$,
and all laminations carried by $B_\alpha$ are assumed to lie in $N(B_\alpha)$
and transverse to its interval fibers.

We now explain that Theorem~\ref{thm: construction from branched surface}
still applies after replacing $B'_\alpha$ by $B_\alpha$.
For each $i$, isotope $\alpha^+_i$ slightly near its two endpoints to 
a properly embedded arc $\wh\alpha^+_i$ such that 
$\wh\alpha^+_i$ has the same endpoints as $\alpha^+_i$, 
while $\Int(\wh\alpha^+_i)$ remains contained in the component of
$U_i - \alpha_i$ on the right side of $\alpha_i$ and remains positively
transverse to $\Fs$.
Let $\wh B_\alpha$ be the branched surface obtained by applying
Construction~\ref{branched surface} with
$\wh\alpha^+_i \times \{1\}$ chosen as the upper arc of the product disk 
associated to $\alpha_i$ for each $i$, 
and let $\wh D(\alpha_i)$ denote the corresponding
product disk of $\wh B_\alpha$.
Then Theorem \ref{thm: construction from branched surface}
applies to $\wh B_\alpha$.
Since each $U_i$ is disjoint from
$\bigcup_{j=1}^n \varphi^{-1}(\partial \alpha_j)$,
the pair
$(\bigcup_{i=1}^n \varphi(\widehat{\alpha}^+_i),
\bigcup_{i=1}^n \alpha_i)$
is isotopic to the pair
$(\bigcup_{i=1}^n \varphi(\alpha^+_i),
\bigcup_{i=1}^n \alpha_i)$
through an isotopy with the endpoints of all arcs remaining on $\partial \Sigma$.
Thus $\wh B_\alpha$ can be isotoped to $B_\alpha$ by moving
the upper arc
$\widehat{\alpha}^+_i \times \{1\}$ of $\wh D(\alpha_i)$
back to $\alpha^+_i \times \{1\}$ for each $i$,
without producing any new triple points.
Hence every lamination produced by
Theorem \ref{thm: construction from branched surface} for $\wh B_\alpha$
can be pushed forward along this isotopy to a lamination fully carried by
$B_\alpha$, 
which intersects each boundary component in simple closed curves
of the same slope.
As in Theorem~\ref{thm: construction from branched surface},
this lamination extends to a foliation on $M$ whose restriction to each
component of $\partial M$ is a foliation by simple closed curves of 
the same slope as above, 
and this foliation then extends to 
co-orientable taut foliations in the corresponding Dehn fillings. 
Thus we obtain the following proposition.

\begin{prop}\label{prop: B_alpha}
The conclusion of
Theorem~\ref{thm: construction from branched surface}
remains valid after replacing $B'_\alpha$ by $B_\alpha$.
\end{prop}


\subsection{Fried's surgery for the suspension flow}\label{subsec: Fried's surgery}

In this subsection,
we verify that each Dehn filling considered in Theorem~\ref{foliation} 
admits a pseudo-Anosov flow obtained from the suspension flow $\psi$ 
by Fried's surgery,
and that the laminations produced in Proposition~\ref{prop: B_alpha}
extend to co-orientable taut foliations transverse to 
the resulting pseudo-Anosov flows in the Dehn fillings.

Recall that the degeneracy slope of $T_i$ is denoted by $\delta_{T_i}$,
and that the degeneracy locus is written as
\[d(T_i)=n_i \cdot \delta_{T_i},\]
where $n_i$ is the multiplicity of $d(T_i)$.
By the convention in Subsection~\ref{subsec: convention},
each slope on $T_i$ represented by a simple closed curve
corresponds to an element of $\Q \cup \{\infty\}$.
In this way, $d(T_i)$ is identified with
a pair of integers $(p_i;q_i)$ such that
$\frac{p_i}{q_i} = \delta_{T_i}$,
$p_i > 0$, and $\gcd(p_i,q_i)$ equals $n_i$.
Choose a component $C_i$ of $\partial \Sigma$ with
$C_i \times \{0\} \subseteq T_i$.
Following Convention \ref{convention on filling multislopes}, 
set
\[c_i = \min \{ k \in \N_+ \mid \varphi^{k}(C_i) = C_i \}.\]
Let $J_i$ be the open interval in $\R \cup \{\infty\}$ between
$\frac{p_i}{q_i \pm c_i}$ not containing $\frac{p_i}{q_i}$.

Here we record a consequence for $p_i$, $q_i$, and $c_i$,
which will be used below.
Note that $d(T_i)$ is the union of $n_i$ parallel curves
of degeneracy slope.
After choosing consistent orientations on these curves,
$d(T_i)$ has algebraic intersection numbers $p_i$ and $q_i$
with the longitude and the meridian on $T_i$, respectively.
Thus there are $p_i$ boundary singularities $u_1,\ldots,u_{p_i}$ of $\Fs$ on $C_i$,
ordered cyclically along $C_i$, and $\varphi^{c_i}$ acts on
the set $\{u_1,\ldots,u_{p_i}\}$ by a shift by $q_i$; 
namely, after possibly reversing the cyclic order of $u_1,\ldots,u_{p_i}$,
$\varphi^{c_i}(u_j)=u_{j+q_i}$ with indices taken modulo $p_i$.
The number $p_i$ of boundary singularities 
is even and nonzero since $\Fs$ is co-orientable.
Since $\varphi$ reverses the co-orientation of $\Fs$, 
the return map $\varphi^{c_i}$ preserves the co-orientation if and only if 
$c_i$ is even. 
It follows that $q_i\equiv c_i \pmod 2$.

\begin{LEM}\label{lem: always pseudo-Anosov flow}
Under the assumptions of Definition \ref{def: degeneracy locus},
suppose further that $\Fs$ is co-orientable.
Let $i \in \{1,\ldots,r\}$ and let $s_i$ be a rational slope on $T_i$ with
$s_i \in J_i$.
Then $\Delta(s_i,d(T_i))\geqslant 2$.
\end{LEM}

\begin{proof}
The case $\Delta(s_i,d(T_i))=0$ occurs only when
$s_i=\delta_{T_i}=\frac{p_i}{q_i}$, which is not in $J_i$.
It remains to rule out the case $\Delta(s_i,d(T_i))=1$.
We argue by contrapositive:
suppose that $\Delta(s_i,d(T_i))=1$.
We will show that $s_i\notin J_i$.

Let $f:\R \to (\R \cup \{\infty\}) - \{0\}$ be the function defined by
\[f(x)=\frac{p_i}{x}, \quad x \in \R.\]
Then $f$ is a homeomorphism, and by definition
\[J_i = f\bigl(\R - [q_i-c_i, q_i+c_i]\bigr) \cup \{0\}.\]
Set $E_i = f([q_i-1,q_i+1])$.
Since $c_i \geqslant 1$, we have
$[q_i-1,q_i+1] \subseteq [q_i-c_i,q_i+c_i]$ and hence
\[E_i \subseteq (\R \cup \{\infty\}) - J_i.\]
Therefore it suffices to prove that $s_i \in E_i$.

Write $u_i = \frac{p_i}{n_i}$, $v_i = \frac{q_i}{n_i}$.
Then
$\delta_{T_i} = \frac{u_i}{v_i}$ and $\gcd(u_i,v_i)=1$.
Write $s_i = \frac{a_i}{b_i}$ with $a_i,b_i \in \Z$ and
$\gcd(a_i,b_i)=1$, and choose the representative so that $a_i \geqslant 0$.
Since $\Delta(s_i,d(T_i))=1$, we have
\[1 = n_i \left| \det
\begin{pmatrix}
u_i & a_i\\
v_i & b_i
\end{pmatrix} \right|
= \left| \det
\begin{pmatrix}
p_i & a_i\\
q_i & b_i
\end{pmatrix} \right|
= |p_i b_i - q_i a_i|.\]

We first show that $a_i \ne 0$.
Indeed, if $a_i=0$, then $\gcd(a_i,b_i)=1$ implies $|b_i|=1$, and so
\[1=\Delta(s_i,d(T_i)) = |p_i|.\]
This contradicts the fact that $p_i$ is a positive even number.

Thus $a_i>0$. From $|p_i b_i - q_i a_i| = 1$,
it follows that
\[-a_i \leqslant -1 \leqslant p_i b_i - q_i a_i \leqslant 1 \leqslant a_i,\]
and hence
\[(q_i-1)a_i \leqslant p_i b_i \leqslant (q_i+1)a_i.\]
Dividing by $a_i$, we obtain
$q_i-1 \leqslant \frac{p_i b_i}{a_i} \leqslant q_i+1$.
Therefore,
\[s_i=\frac{a_i}{b_i} = f(\frac{p_i b_i}{a_i}) \in f([q_i-1,q_i+1]) = E_i.\]
Since $E_i \subseteq (\R \cup \{\infty\}) - J_i$, we conclude that
$s_i \notin J_i$.
This proves the contrapositive, and hence the claim holds.
\end{proof}

Combining Lemma \ref{lem: always pseudo-Anosov flow} with 
Construction \ref{Fried},
we obtain the following consequence.

\begin{COR}\label{cor: always pA}
For each $i$, fix a slope $s_i \in J_i \cap (\Q \cup \{\infty\})$,
and let $\s$ denote the multislope $(s_1,\ldots,s_r)$.
Then the suspension flow $\psi$ of $M$ induces
a pseudo-Anosov flow on $M(\s)$.
\end{COR}

The resulting pseudo-Anosov flow on $M(\s)$
always contains singular orbits, except in the following special case.

\begin{rmk}
Suppose that $\Delta(s_i,d(T_i))=2$ for some $s_i\in J_i$.
We continue to write $s_i=\frac{a_i}{b_i}$ with $\gcd(a_i,b_i)=1$ and 
$a_i\geqslant 0$,
and use the same function
$f: \mathbb R\to (\mathbb R\cup\{\infty\})-\{0\}$ from the previous proof.
Below,
we consider separately the three cases $a_i\geqslant 2$, $a_i=1$, and $a_i=0$.

\begin{enumerate}
    \item 
We first assume $a_i \geqslant 2$.
Then
\[a_i \geqslant 2 = \Delta(s_i,d(T_i)) = |p_i b_i - q_i a_i|.\]
As in the proof of Lemma~\ref{lem: always pseudo-Anosov flow},
this also implies that
$q_i - 1 \leqslant \frac{p_i b_i}{a_i} \leqslant q_i + 1$,
and hence
$s_i = f(\frac{p_i b_i}{a_i}) \in f([q_i-1,q_i+1]) \subseteq 
(\R \cup \{\infty\}) - J_i$,
a contradiction.

\item 
Suppose that $a_i=1$.
Then $2=|p_i b_i-q_i a_i|=|p_i b_i-q_i|$,
which implies that $q_i - 2 \leqslant p_ib_i \leqslant q_i + 2$.
Since $2 \mid p_i$, the integer $q_i$ must be even.
This implies $2\mid c_i$, and so $c_i\geqslant 2$.
Hence $p_ib_i \in [q_i-2,q_i+2]\subseteq [q_i-c_i,q_i+c_i]$.
Therefore, since $a_i = 1$,
we have
\[s_i = f(p_i b_i)
   \in f([q_i-c_i,q_i+c_i])
   =(\R\cup\{\infty\})-J_i,\]
again a contradiction.

\item
Suppose that $a_i=0$.
Then $|b_i|=1$, and
$2=\Delta(s_i,d(T_i))=|p_i b_i|=p_i$.
Thus $p_i=2$, and by the convention in Subsection~\ref{subsec: convention},
we have $q_i=1$. 
Therefore, the only possibility is that $s_i = 0$ and $(p_i;q_i)=(2;1)$.
\end{enumerate}

Consequently, the pseudo-Anosov flow in $M(\s)$ induced by $\psi$
has no singular orbit only if case (3) occurs for every boundary component $T_i$ and
$\psi$ has no singular orbit in $\Int(M)$.
This means that $\Fs$ has no interior singularities,
every boundary component of $M$ has degeneracy locus $(2;1)$,
and the restriction of $\s$ to each boundary component is the longitude.
By the Euler-Poincar\'e formula,
this occurs only when the fibered surface has genus one.
\end{rmk}

Finally, we record the canonical form of 
the construction of $\F_\alpha(\s)$ used in this paper.

\begin{prop}\label{prop: construction from the new branched surface}
Let $\alpha = (\alpha_1,\ldots,\alpha_n)$ be an admissible system of arcs
with respect to $(\Sigma,\varphi)$,
and let $B_\alpha$ be the branched surface given in
Construction \ref{new branched surface}.
Let $s_i \in J_i \cap (\Q \cup \{\infty\})$
for $1 \leqslant i \leqslant r$,
and let $\s = (s_1,\ldots,s_r)$.

\textnormal{(a)}
The branched surface $B_\alpha$ fully carries an essential lamination $\Ll_\alpha$
that extends to a foliation $\F_\alpha$ transverse to $\psi$,
which intersects each $T_i$ in simple closed curves of slope $s_i$.

\textnormal{(b)}
There exists a pseudo-Anosov flow $\phi$ on $M(\s)$
obtained from Fried's surgery on the suspension flow $\psi$ of $M$,
and $\F_\alpha$ extends to a co-orientable taut foliation
$\F_\alpha(\s)$ of $M(\s)$ transverse to $\phi$.
\end{prop}

\begin{proof}
As in Proposition \ref{prop: B_alpha},
Theorem \ref{thm: construction from branched surface} applies with
$B_\alpha$ in place of $B'_\alpha$.
Thus $B_\alpha$ fully carries an essential lamination $\Ll_\alpha$
which intersects each $T_i$ in a union of parallel simple closed curves
of slope $s_i$.
Moreover, $\Ll_\alpha$ extends to a foliation $\F_\alpha$ of $M$,
and $\F_\alpha$ further extends to a co-orientable taut foliation
$\F_\alpha(\s)$ of $M(\s)$.

Recall that we have fixed a fibered neighborhood $N(B_\alpha)$ of $B_\alpha$ such that
every interval fiber of $N(B_\alpha)$ lies in an orbit of $\psi$,
and every lamination carried by $B_\alpha$ is transverse to the interval
fibers of $N(B_\alpha)$.
Since $\Ll_\alpha$ is carried by $B_\alpha$, it is transverse to the
interval fibers of $N(B_\alpha)$ and therefore transverse to $\psi$.
Note that each complementary region of $N(B_\alpha)$ in $M$ is homeomorphic to
a surface bundle over an interval, with the surface fibers transverse to
$\psi$.
Hence the extension of $\Ll_\alpha$ through the complementary regions
can be chosen so that the resulting foliation $\F_\alpha$ 
is transverse to $\psi$.

By Corollary~\ref{cor: always pA}, Fried's surgery produces 
a pseudo-Anosov flow $\phi$ on $M(\s)$ from the suspension flow $\psi$ on $M$.
Since each slope $s_i$ is distinct from the degeneracy slope on $T_i$,
the foliation $\F_\alpha$ extends over the filling solid tori to
$\F_\alpha(\s)$ in such a way that the transversality to the flow is
preserved.
Therefore $\F_\alpha(\s)$ is transverse to $\phi$.
\end{proof}

\section{Producing $\R$-covered foliations}\label{sec: R-covered}

We work under the setting of Theorem~\ref{foliation}.
Let $\Sigma$ be a compact orientable surface,
$\varphi: \Sigma \to\Sigma$
an orientation-preserving pseudo-Anosov homeomorphism with
invariant foliations $\Fs, \Fu$,
and 
\[M=\Sigma \times I / \stackrel{\varphi}{\sim}\]
the mapping torus of $\varphi$ with boundary components $T_1,\ldots,T_r$.
We assume that $\Fs$ is co-orientable and $\varphi$ reverses its co-orientation.
For each $i$ let $J_i$ be the interval of slopes on $T_i$
specified in Theorem~\ref{foliation}.
Let $\psi$ be the suspension flow of $\varphi$ on $M$.

We prove Theorem \ref{main} (a) in this section.
 
\begin{main (a)}
For each $1 \leqslant i \leqslant r$, choose a slope
$s_i \in J_i \cap (\Q \cup \{\infty\})$,
and let $\s=(s_1,\ldots,s_r)$ be the corresponding multislope.
Then there exists an admissible system of arcs
$\alpha^*$ for $(\Sigma,\varphi)$
such that the induced foliation $\F_{\alpha^*}(\s)$ is $\R$-covered.
\end{main (a)}

We fix a co-orientation on $\Fs$.
We then choose a continuously-varying leafwise orientation on $\Fs$
so that it points from the right side of any positively oriented
transversal of $\Fs$ to the left side.
Accordingly, we endow $\Fu$ with the co-orientation
compatible with the chosen leafwise orientation on $\Fs$.

\subsection{A refined admissible system of arcs}\label{subsec: refined system}

In this subsection,
we construct an admissible system of arcs 
$\alpha = (\alpha_1,\ldots,\alpha_n)$ on $\Sigma$
with the additional property that
each $\alpha_i$ is simultaneously transverse to $\Fs$ and $\Fu$.

Each prong of $\Fu$ is naturally parametrized by $[0,+\infty)$,
with the parameter starting at a boundary singularity.
Let
\[\beta'_1,\ldots,\beta'_n \colon [0,+\infty) \to \Sigma\]
denote the prongs of $\Fu$ whose parametrizations are consistent with
the co-orientation on $\Fs$, and let
\[\gamma'_1,\ldots,\gamma'_n \colon [0,+\infty) \to \Sigma\]
denote those whose parametrizations are opposite to
the co-orientation on $\Fs$.
We first construct a properly embedded arc on $\Sigma$ from
$\beta'_i(0)$ to $\gamma'_i(0)$ by adding a segment joining them
that is positively transverse to $\Fu$.

\begin{figure}
	\centering
	\subfigure[]{
	\includegraphics[width=0.45\textwidth]{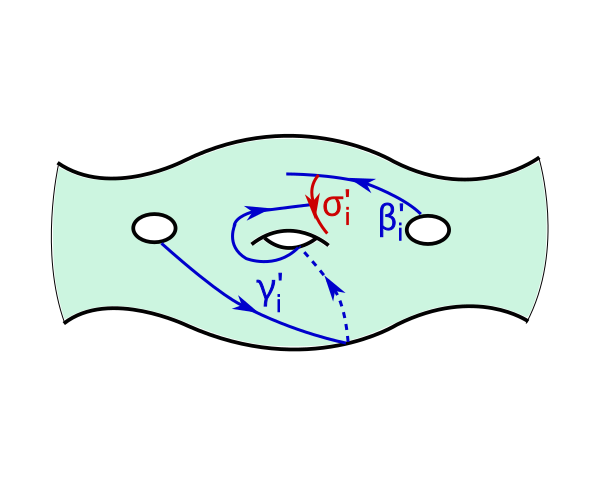}}
\subfigure[]{
	\includegraphics[width=0.45\textwidth]{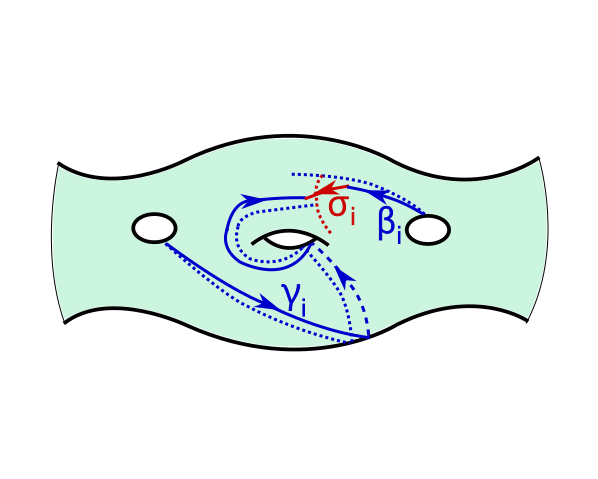}}
	\caption{(a) shows the path $\sigma'_i$ 
    joining the rays $\beta'_i$ and $\gamma'_i$.
(b) shows how these segments are modified to $\beta_i$, $\gamma_i$, $\sigma_i$;
the dashed curves represent $\beta'_i$, $\gamma'_i$, $\sigma'_i$.
    }\label{fig: refined arcs}
\end{figure}

\begin{cons}\rm
For each $1 \leqslant i \leqslant n$, we choose $r_i \in \R_+$ and a segment
$\sigma'_i \colon [0,1] \to \Sigma$
such that 
\begin{enumerate}[(i)]
    \item $\sigma'_i(0) = \beta'_i(r_i)$, and
$\sigma'_i([0,1])$ is disjoint from $\beta'_i([0,r_i])$
except at the point $\sigma'_i(0)$.    

    \item The image $\sigma'_i([0,1])$ is contained in a single leaf of $\Fs$ and 
is disjoint from all singularities of $\Fs$, 

    \item The increasing orientation on $[0,1]$ 
is compatible with the co-orientation on $\Fu$ via $\sigma'_i$.
\end{enumerate}
Since $\gamma'_i$ has dense image in $\Sigma$, 
there exists $t_i \in \R_+$ such that
$\gamma'_i(t_i) \in \sigma'_i([0,1])$ and 
$\gamma'_i([0,t_i)) \cap \sigma'_i([0,1]) = \emptyset$.
Let $w_i \in [0,1]$ be such that
$\sigma'_i(w_i) = \gamma'_i(t_i)$.
Note that $w_i \ne 0$, since $\beta'_i$ and $\gamma'_i$ 
are distinct half-leaves of $\Fu$.
Define a broken path
\[\alpha''_i = \beta'_i\mid_{[0,r_i]} * 
\sigma'_i\mid_{[0,w_i]} * 
\overline{\gamma'_i}\mid_{[0,t_i]},\]
where $*$ denotes concatenation of paths and 
$\overline{\gamma'_i}$ denotes $\gamma'_i$ with reversed orientation.
Note that $\alpha''_i$ is an embedding, since 
$\gamma'_i([0,t_i)) \cap \sigma'_i([0,1]) = \emptyset$.
\end{cons}

Since $\beta'_i$ is positively transverse to $\Fs$ and 
$\gamma'_i$ is negatively transverse to $\Fs$, 
the orientations on $\beta'_i, \gamma'_i$ induced by 
the increasing parameterization determine 
opposite normal directions along $\sigma'_i$. 
See Figure \ref{fig: refined arcs} (a) for an illustration.

We now isotope each $\alpha''_i$ as follows
to obtain an embedded path $\alpha'_i$
that is positively transverse to both $\Fs$ and $\Fu$,
and then smooth all intersections among $\alpha'_1, \ldots, \alpha'_n$ 
to obtain a collection of disjoint arcs $\alpha_1,\ldots,\alpha_n$.

\begin{figure}
	\centering
	\subfigure[]{
	\includegraphics[width=0.45\textwidth]{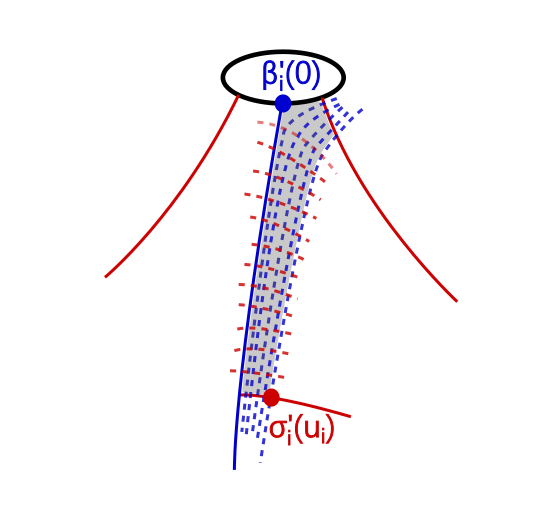}}
\subfigure[]{
	\includegraphics[width=0.45\textwidth]{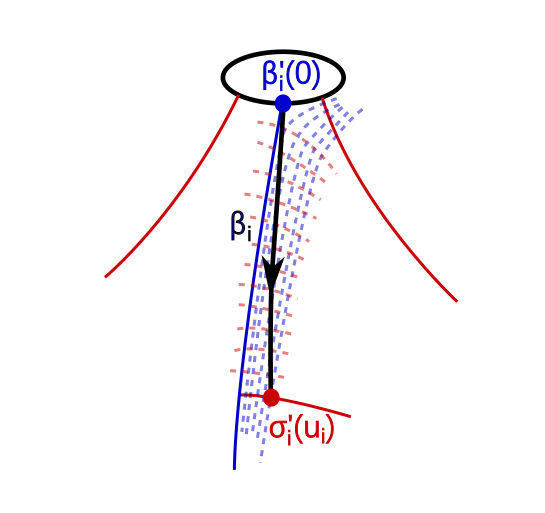}}
	\caption{In (a), for $u_i \in (0,w_i)$ sufficiently small,
the points $\beta'_i(0)$ and $\sigma'_i(u_i)$
lie in the closure of an open product chart
$U \cong (0,1) \times (0,1)$ for $(\Fs,\Fu)$.
As shown in (b), a path $\beta_i \subseteq U$
joining $\beta'_i(0)$ to $\sigma'_i(u_i)$
is positively transverse to both $\Fs$ and $\Fu$.
}\label{fig: open product chart}
\end{figure}

\begin{cons}\rm
For each $1 \leqslant i \leqslant n$, we first modify the three paths
$\beta'_i, \gamma'_i, \sigma'_i$ as follows
to make them positively transverse to both $\Fs$ and $\Fu$.

\begin{enumerate}
\item
Since $\beta'_i([0,r_i])$ is disjoint from the interior singularities of $\Fs$,
there exists $u_i \in (0,w_i)$ such that the points
$\beta'_i(0)$ and $\sigma'_i(u_i)$ lie in the closure of
some open product chart of $(\Fs,\Fu)$ (see Figure \ref{fig: open product chart} (a)).
Hence they can be joined by a path
\[\beta_i \colon [0,1] \to \Sigma\]
with $\beta_i(0)=\beta'_i(0)$ and $\beta_i(1)=\sigma'_i(u_i)$,
which is positively transverse to both $\Fs$ and $\Fu$
(see Figure \ref{fig: open product chart} (b)).

\item
Similarly, there exists $v_i \in (u_i,w_i)$ such that the points
$\gamma'_i(0)$ and $\sigma'_i(v_i)$ can be joined by a path
\[\gamma_i \colon [0,1] \to \Sigma\]
from $\gamma'_i(0)$ to $\sigma'_i(v_i)$ which is negatively transverse
to both $\Fs$ and $\Fu$.

\item
Since the segment $\sigma'_i([u_i,v_i])$ is disjoint from the singularities of $\Fs$,
there exists an open product chart $U$ of $(\Fs,\Fu)$ containing this segment.
Hence there exist $\epsilon_1,\epsilon_2 \in (0,1)$ sufficiently small such that
\[\beta_i([1-\epsilon_1,1]) \subseteq U,
\qquad
\gamma_i([1-\epsilon_2,1]) \subseteq U.\]
Thus there exists a path
\[\sigma_i: [0,1] \to \Sigma\]
from $\beta_i(1-\epsilon_1)$ to $\gamma_i(1-\epsilon_2)$
which is positively transverse to both $\Fs$ and $\Fu$.
\end{enumerate}
Define
\[\alpha'_i
= \beta_i\mid_{[0,1-\epsilon_1]} * 
\sigma_i * 
\overline{\gamma_i}\mid_{[0,1-\epsilon_2]}.\]
Then $\alpha'_i$ is an embedded path positively transverse to both $\Fs$ and $\Fu$;
see Figure~\ref{fig: refined arcs} (b)
for the modification of $\alpha''_i$ to $\alpha'_i$.
Note that each $\alpha'_i$ remains disjoint from the singularities of $\Fs$.
By constructing the paths $\alpha'_i$ successively,
we can ensure that they intersect only in finitely many transverse double points.
At each intersection point of $\alpha'_i$ and $\alpha'_j$ for distinct $i,j$,
we smooth the crossing in a manner consistent with the orientations of the paths
(see Figure \ref{fig: smooth}).
This produces a collection of $n$ disjoint oriented properly embedded arcs,
together with a (possibly empty) collection of oriented circles.
Since $\alpha'_1,\ldots,\alpha'_n$ are positively transverse to both $\F^s$ and $\F^u$,
every resulting arc and circle remains positively transverse to
both $\F^s$ and $\F^u$.
Discarding the circle components,
we denote the remaining $n$ properly embedded arcs by
$\alpha_1,\ldots,\alpha_n$.
Finally,
we slightly isotope each $\alpha_i$ near its endpoints so that 
$\partial \alpha_i$ is disjoint from the boundary singularities of $\Fs$ and $\Fu$,
$\partial \alpha_i \cap \varphi(\partial \alpha_j) = \emptyset$
for all distinct $i,j$,
and each $\alpha_i$ remains positively transverse to both $\Fs$ and $\Fu$.
\end{cons}


\begin{figure}
	\centering
	\subfigure[]{
	\includegraphics[width=0.3\textwidth]{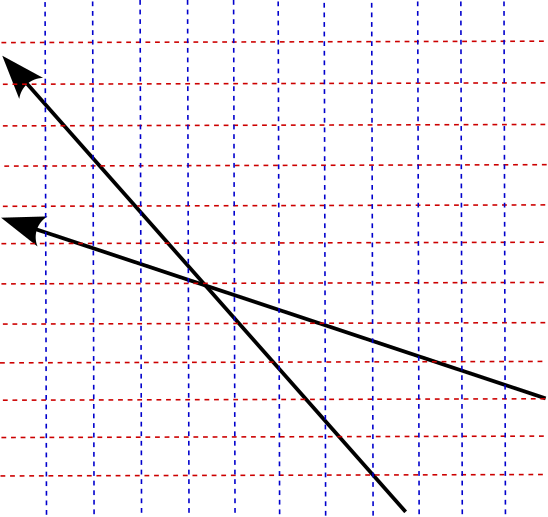}}
\subfigure[]{
	\includegraphics[width=0.3\textwidth]{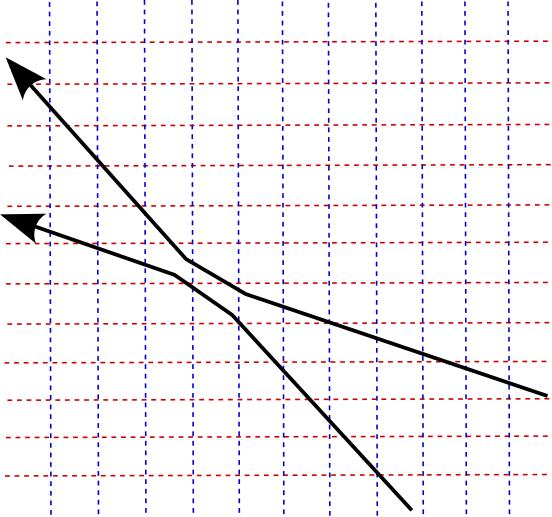}}
	\caption{Smoothing the double intersection of two paths
positively transverse to both $\Fs$ and $\Fu$.}\label{fig: smooth}
\end{figure}

Now the collection $\alpha_1,\ldots,\alpha_n$ satisfies Definition~\ref{system of arcs};
let $\alpha$ be the admissible system of arcs $(\alpha_1,\ldots,\alpha_n)$.
Since the endpoints of each $\alpha_i$ have been isotoped to be disjoint from
the boundary singularities of $\Fs$ and $\Fu$,
we may perturb $\alpha_i$ slightly to a parallel arc $\alpha^+_i$ on its right side,
so that $\alpha^+_i$ satisfies Construction~\ref{new branched surface} and
is also positively transverse to $\Fu$.
As in Construction~\ref{new branched surface},
a branched surface $B_\alpha$
is constructed from the admissible system of arcs $\alpha$ by
adding product disks $D(\alpha_i)$ 
with lower arc $\alpha_i \times \{0\}$ and
upper arc $\varphi(\alpha^+_i) \times \{1\}$.
As in Subsection~\ref{subsec: perturb}, 
we may isotope each $D(\alpha_i)$ to be transverse to the suspension flow $\psi$.
Proceeding as in 
Proposition~\ref{prop: construction from the new branched surface} (a),
we obtain a foliation $\F_\alpha$ in $M$ intersecting
each $T_i$ in simple closed curves of slope $s_i$,
which is transverse to $\psi$.

\subsection{The intersection behavior of $\F_\alpha$ and the weak stable foliation}\label{subsec: F and Fws}

We write the weak stable and weak unstable foliations 
$\Fws(\psi)$ and $\Fwu(\psi)$ of $\psi$ simply as $\Fws$ and $\Fwu$.
Let $p: \w M \to M$ be the universal cover of $M$.
We denote by
$\w{\F_\alpha}, \w \psi, \w{\Fws}, \w{\Fwu}$
the lifts to $\w M$ of
$\F_\alpha, \psi, \Fws, \Fwu$,
respectively.

In this subsection,
we analyze the intersections between $\w{\F_\alpha}$ and $\w \Fws$
and prove the following proposition.

\begin{prop}\label{prop: l and Fws transverse}
Let $l$ be a regular leaf or a half-leaf of $\w{\Fws}$.
Let $x \in l$ be an arbitrary point,
and let $\lambda$ be a leaf of $\w{\F_\alpha}$ that intersects $l$.
Then the orbit $\w \psi(x)$ intersects $\lambda$ exactly once.
\end{prop}

The proof relies only on the property that each $\alpha_i$
is positively transverse to $\Fs$.
The condition that each $\alpha_i$ is positively transverse to $\Fu$
is not used here and yields a symmetric statement.
See Remark \ref{rmk: only Fs} and 
Proposition \ref{prop: l and Fwu transverse} for details.


The projection of $M = (\Sigma \times I)/ \stackrel{\varphi}{\sim}$ 
onto the second coordinate induces a fibration
$M \to S^1$ with fiber $\Sigma$.
We fix the orientation on the base $S^1$ so that it agrees with the increasing orientation of the second coordinate on $\Sigma \times I$.
This fibration lifts to a fibration of $\w M$ over $\R$, 
whose fibers are copies of the universal cover $\w \Sigma$ of $\Sigma$.
Under this identification, 
the $\w \Sigma$-fibers of $\w M$ are indexed by
\[\w \Sigma \times \{t\}, \quad t \in \R,\]
so that $p(\w \Sigma \times \{t\}) = \Sigma \times \{0\}$
when $t \in \Z$, and the increasing orientation on the base $\R$ agrees with
the chosen orientation on the base $S^1$ of $M$.

\begin{figure}
	\centering
	\subfigure[]{
	\includegraphics[width=0.45\textwidth]{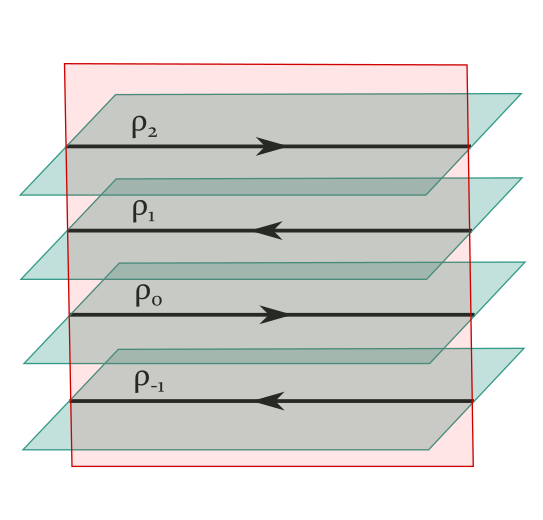}}
    \subfigure[]{
	\includegraphics[width=0.45\textwidth]{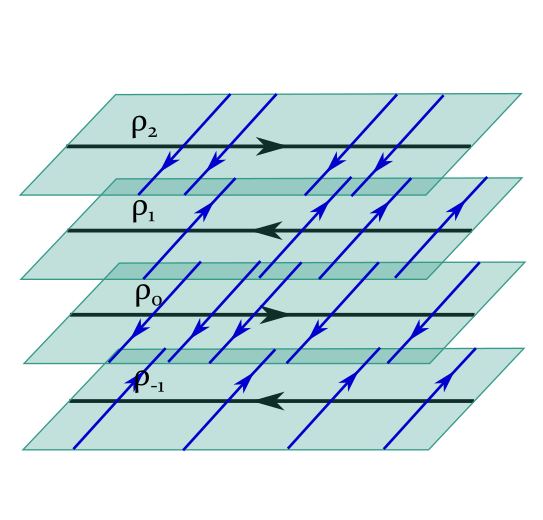}}
    \subfigure[]{
	\includegraphics[width=0.45\textwidth]{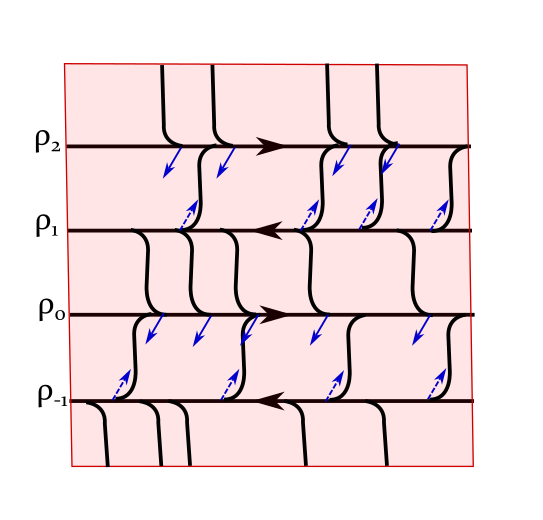}}   
    \subfigure[]{
	\includegraphics[width=0.45\textwidth]{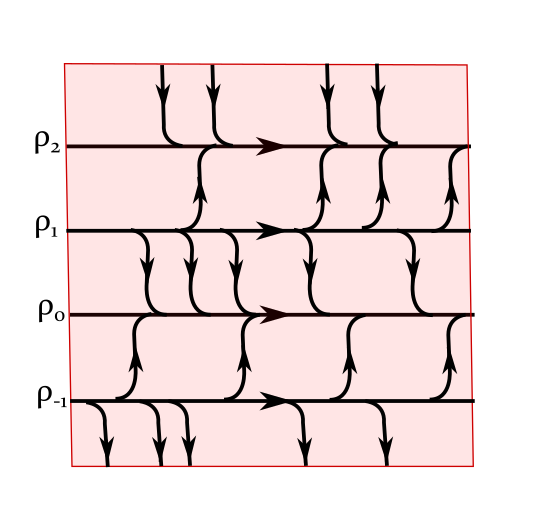}}    
    \caption{(a) describes the orientation on each $\rho_k$ induced by
the leafwise orientation on $\w \Fs \times \{k\}$,
with orientation alternating as $k$ varies.
(b) describes the intersections between $\w \alpha \times \{k\}$ and $\rho_k$,
where the blue arcs represent $\w \alpha \times \{k\}$
and the arrows indicate their orientations,
which give alternating transverse orientations of $\rho_k$.
(c) is a picture of the train track $\tau_l$,
where the arrows indicate the normal orientations induced by
the arcs $\w \alpha \times \{k\}$.
(d) describes the orientation on the train track $\tau_l$.
}\label{fig: intersection train track to Fws}
\end{figure}

Let $l$ be a regular leaf or a half-leaf of a singular leaf of $\w{\Fws}$, and let
\[\rho_k = (\w \Sigma \times \{k\}) \cap l, \quad k \in \Z.\]
We orient each $\rho_k$ so that it is positively transverse to 
$\w \Fu \times \{k\}$. 
Since $\varphi$ is co-orientation-reversing,
it sends positively oriented transversals of $\Fs$ on $\Sigma$
to negatively oriented ones.
Hence the orientation of $\rho_k$ is opposite to the orientation of $\rho_{k+1}$.
This means that, if we choose a continuous orientation for the family
$\{l \cap (\w \Sigma \times \{t\}) \mid t \in \R\}$,
then one of $\rho_k$ and $\rho_{k+1}$ agrees with this orientation,
while the other is oppositely oriented.
See Figure~\ref{fig: intersection train track to Fws} (a).

We fix a co-orientation on $l$ which agrees with
the co-orientation of $\rho_0$ in $\w \Sigma \times \{0\}$ 
induced by $\w \Fs \times \{0\}$.
Since $\varphi$ reverses the co-orientation of $\Fs$,
for each $k$ the co-orientation of $\rho_k$ induced by $\w \Fs \times \{k\}$
is opposite to that of $\rho_{k+1}$ induced by $\w \Fs \times \{k+1\}$.
Hence the co-orientation of $\rho_k$ induced by $\w \Fs \times \{k\}$
agrees with the co-orientation on $l$ if and only if $2 \mid k$.
It follows that the pullbacks of $\alpha_1,\ldots,\alpha_n$ to $\w M$ 
are positively transverse to $l$ when they are contained in
$\w \Sigma \times \{k\}$ with $2 \mid k$, and negatively transverse to $l$
when they are contained in $\w \Sigma \times \{k\}$ with $2 \nmid k$.
See Figure~\ref{fig: intersection train track to Fws} (b).

Let $\w{B_\alpha}$ be the pullback of the branched surface $B_\alpha$ to $\w M$,
and let $\tau_l$ be the train track given by $l \cap \w{B_\alpha}$.
Now consider a product disk $D(\alpha_i)$ of $B_\alpha$ induced by $\alpha_i$.
Let $\w{D}$ be a lift of $D(\alpha_i)$ to $\w M$, 
and suppose that $\w{D}$ intersects $l$.
Let $k \in \Z$ be such that $\w{D}$ intersects both $\rho_k$ and $\rho_{k+1}$,
and let $\eta^-$ and $\eta^+$ denote the arcs
\[\w{D} \cap (\w \Sigma \times \{k\}), \quad
\w{D} \cap (\w \Sigma \times \{k+1\}),\]
respectively.
Since $\eta^\pm$ are lifts of the lower arc
$\alpha_i \times \{0\}$ and the upper arc $\varphi(\alpha^+_i) \times \{0\}$
of $D(\alpha_i)$ to $\w M$,
we associate $\eta^\pm$ with the orientations induced by
$\alpha_i$ and $\varphi(\alpha^+_i)$, respectively.
As the arcs $\alpha_i, \alpha^+_i$ are transverse to $\Fs$
and the product disk $D(\alpha_i)$ is transverse to $\psi$,
$\eta^-$ intersects $\rho_k$ exactly once,
$\eta^+$ intersects $\rho_{k+1}$ exactly once,
and that $\w D$ intersects $l$ in a connected arc.
Recall that the cusp direction at the lower (resp.\ upper) arc of $D(\alpha_i)$ 
points to the left (resp.\ right) in $\Sigma \times \{0\}$.
Moreover, as noted at the beginning of this section, 
since $\rho_k$ is positively transverse to $\w \Fu \times \{k\}$, 
its orientation points to the left of 
every positively oriented transversal of $\w \Fs \times \{k\}$ that it meets.
Since $\eta^-$ is positively transverse to $\w \Fs \times \{k\}$, 
the cusp direction of $\tau_l$ at $\eta^- \cap \rho_k$ 
agrees with the orientation of $\rho_k$.
On the other hand, 
the orientations of $\rho_k$ and $\rho_{k+1}$ are opposite,
and the cusp directions of $\tau_l$ at 
$\eta^- \cap \rho_k$ and $\eta^+ \cap \rho_{k+1}$
also point toward opposite sides of $\rho_k$ and $\rho_{k+1}$, respectively.
Hence the cusp direction at $\eta^+ \cap \rho_{k+1}$ 
still agrees with the orientation of $\rho_{k+1}$.
See Figure~\ref{fig: intersection train track to Fws} (c) for an example.

Let $\w{N(B_\alpha)}$ denote the pullback of $N(B_\alpha)$ to $\w M$,
and let $N(\tau_l) = \w{N(B_\alpha)} \cap l$.
As in the previous subsection,
each interval fiber of $N(\tau_l)$ is contained in an orbit of $\w \psi$
and is transverse to $\w \F$.
Thus $N(\tau_l)$ inherits a natural $I$-bundle structure from $\w{N(B_\alpha)}$.
Let $\pi_l \colon N(\tau_l) \to \tau_l$ denote the projection that
collapses each interval fiber to a single point.
Throughout, whenever a curve on $l$ is said to be carried by $\tau_l$,
we assume that it is transverse to the interval fibers of $N(\tau_l)$.

\begin{LEM}\label{lem: intersects at most two}
Suppose that a real line $\eta \colon \R \to l$ is carried by $\tau_l$.
Then $\pi_l(\eta(\R))$ intersects at most two components of
$\bigcup_{k \in \Z} \rho_k$,
and intersects $\w \psi(x)$ for every $x \in l$.
\end{LEM}
\begin{proof}
Let $\tau_l$ be oriented such that
\begin{enumerate}
    \item The induced orientation of $\tau_l$ on $\rho_k$ is consistent with 
the previously chosen orientation on $\rho_k$ if $2 \mid k$ and 
is opposite to it if $2 \nmid k$.
Note that 
the family $\{(\w \Sigma \times \{t\}) \cap l \mid t \in \R\}$ admits 
a continuous orientation consistent with 
the induced orientation of $\tau_l$ on every $\rho_k$.

    \item For each edge $e$ of $\tau_l$ with endpoints on $\rho_k$ and $\rho_{k+1}$,
the orientation on $e$
goes from $e \cap \rho_k$ to $e \cap \rho_{k+1}$ if $2 \nmid k$ and 
goes from $e \cap \rho_{k+1}$ to $e \cap \rho_k$ if $2 \mid k$.
\end{enumerate}
This determines a well-defined orientation of $\tau_l$; 
see Figure \ref{fig: intersection train track to Fws} (d).

Let $\gamma = \pi_l \circ \eta$. 
We first suppose that $\gamma(\R)$ meets $\rho_k$ for some $k \in \Z$ with $2 \mid k$.
Let $y \in \R$ be such that $\gamma(y) \in \rho_k$.
Since the cusp direction at any triple point of $\tau_l$ on $\rho_k$
agrees with the orientation on $\rho_k$ chosen previously,
it agrees with the orientation on $\tau_l$,
so we must have
\[\gamma([y,+\infty)) \subseteq \rho_k.\]
Similarly, if $\gamma(z) \in \rho_k$ for 
some $z \in \R$ and $k \in \Z$ with $2 \nmid k$,
then the orientation on $\tau_l$ is opposite to 
the previously chosen orientation on $\rho_k$,
and hence is opposite to the cusp direction at all triple points of $\tau_l$ on $\rho_k$,
so
\[\gamma((-\infty,z]) \subseteq \rho_k.\]

Therefore, the image $\gamma(\R)$ can meet at most one component of
$\bigcup_{k \in 2\Z} \rho_k$
and at most one component of
$\bigcup_{k \in \Z - 2\Z} \rho_k$.
It follows that $\gamma(\R)$ meets at most two components of
$\bigcup_{k \in \Z} \rho_k$.

Let $j \in \Z$ be such that $\gamma(\R) \cap \rho_j \ne \emptyset$.
It's not hard to observe that $\gamma(\R)$ 
separates $\bigcup_{k \in \Z, k < j-1} \rho_k$ from 
$\bigcup_{k \in \Z, k > j+1} \rho_k$ in $l$.
Since $\psi$ is the suspension flow of $\varphi$,
each orbit of $\w \psi$ contained in $l$
intersects every $\rho_k$ exactly once.
It follows that $\gamma(\R)$ intersects all orbits of $\w \psi$ contained in $l$.
\end{proof}

We are now ready to prove Proposition~\ref{prop: l and Fws transverse}.

\begin{proof}[Proof of Proposition~\ref{prop: l and Fws transverse}]
Let $\lambda$ be a leaf of $\w{\F_\alpha}$ that intersects $l$,
and let $x \in l$.
If $\lambda \cap l \subseteq N(\tau_l)$,
then Lemma~\ref{lem: intersects at most two} implies that
$\pi_l(\lambda \cap l)$ intersects $\w \psi(x)$,
and hence $\pi_l(y) \in \w \psi(x)$ for some point $y \in \lambda \cap l$.
Since each interval fiber of $N(\tau_l)$ is contained in an orbit of $\w \psi$,
it follows that $y \in \w \psi(x)$.

Now suppose that $\lambda \cap l \nsubseteq N(\tau_l)$.
Since $l$ is transverse to $\w \psi$ and
$\tau_l$ has only bigon complementary regions,
we may enlarge the fibered neighborhood $N(\tau_l)$ to
a (possibly not $\pi_1$-equivariant) fibered neighborhood $N'(\tau_l)$ of $\tau_l$
such that every interval fiber is still contained in some orbit of $\w \psi$,
and $\lambda \cap l$ is contained in $N'(\tau_l)$ and transverse to
its interval fibers.
It still follows from Lemma~\ref{lem: intersects at most two} that
the image of $\lambda \cap l$ under the associated collapsing map
$N'(\tau_l) \to \tau_l$ intersects all orbits of $\w \psi$ contained in $l$,
and hence intersects $\w \psi(x)$.
As in the previous case,
there exists $y \in \lambda \cap l$ such that
the collapsing map sends $y$ to a point in $\w \psi(x)$.
This implies that $y \in \w \psi(x)$.

Thus, $\w \psi(x)$ intersects $\lambda$ nontrivially.
Since $\psi$ is transverse to $\F_\alpha$,
$\w \psi$ is transverse to $\w{\F_\alpha}$,
and hence $\w \psi(x)$ intersects $\lambda$ exactly once.
This completes the proof of Proposition~\ref{prop: l and Fws transverse}.
\end{proof}

\begin{rmk}\label{rmk: only Fs}
In the preceding discussion of this subsection,
we use only the positive transversality of $\alpha_i, \alpha^+_i$ to $\Fs$,
and not the additional condition from Subsection~\ref{subsec: refined system}
that $\alpha_i, \alpha^+_i$ are also positively transverse to $\Fu$.
Hence, if the admissible system of arcs $\alpha$ 
does not satisfy this additional condition,
the argument of Proposition~\ref{prop: l and Fws transverse} still applies.
\end{rmk}

Since $\alpha_i, \alpha^+_i$ are also positively transverse to $\Fu$,
we obtain the analogous statement:

\begin{prop}\label{prop: l and Fwu transverse}
Let $l'$ be a regular leaf or a half-leaf of $\w{\Fwu}$.
Let $y \in l'$,
and let $\lambda$ be a leaf of $\w{\F_\alpha}$ that intersects $l'$.
Then the orbit $\w \psi(y)$ intersects $\lambda$ exactly once.
\end{prop}

\subsection{Verifying the foliation $\F_\alpha(\s)$ is $\R$-covered}\label{subsec: verify R-covered}

By Proposition \ref{prop: construction from the new branched surface} (b),
the foliation $\F_\alpha$ extends to 
a co-orientable taut foliation $\F_\alpha(\s)$ in $M(\s)$,
and Fried's surgery produces a pseudo-Anosov flow $\phi$ on $M(\s)$ from $\psi$,
which is transverse to $\F_\alpha(\s)$.
Let $W$ denote the Dehn filling $M(\s)$,
and let $\F$ denote $\F_\alpha(\s)$.
We write $\Ews = \Fws(\phi)$ and $\Ewu = \Fwu(\phi)$ for
the weak stable and unstable foliations of $\phi$, respectively.
Let $\w W$ be the universal cover of $W$,
and denote by $\w \F, \w \phi, \w \Ews, \w \Ewu$
the pullbacks of $\F, \phi, \Ews, \Ewu$ to $\w W$, respectively.

\begin{prop}\label{regulating}
For any leaf $\lambda$ of $\w \F$ and any point $x \in \w W$,
the orbit $\w \phi(x)$ has nonempty intersection with $\lambda$.
\end{prop}

\begin{proof}
Let \[o: \w W \to \Or(\phi)\]
be the canonical projection sending each orbit of $\w \phi$
to a point in the orbit space $\Or(\phi)$.
Choose a point $y \in \lambda$.
There exists a broken path
\[\sigma = \sigma_1 * \sigma_2 * \ldots * \sigma_k : I \to \Or(\phi)\]
with $\sigma(0) = o(y)$ and $\sigma(1) = o(x)$,
where each $\sigma_i$ is contained in a leaf of
$\Ors(\phi)$ or $\Oru(\phi)$.

Since $W$ is obtained from Fried's surgery on $M$,
the filling solid tori are collapsed to
a finite union of closed orbits of $\phi$ in $W$;
denote this union by $C$, and denote by $\w C$ its pullback to $\w W$.
Note that both the singularities of $\Or(\phi)$ and
the set $o(\w C)$ are discrete in $\Or(\phi)$.
Thus, after possibly subdividing each $\sigma_i$,
we may assume that for every $1 \leqslant i \leqslant k$,
the interior $\Int(\sigma_i)$ contains no singularity of $\Ors(\phi)$
and no point of $o(\w C)$.
It follows that there exists a regular leaf or half-leaf
$l_i$ of $\w \Ews$ or $\w \Ewu$
such that $\sigma_i \subseteq o(l_i)$ and
$\Int(l_i) \subseteq \w W - \w C$.
We regard $\Int(\w M)$ as the universal cover of $\w W - \w C$.
If $l_i$ is a regular leaf of $\w \Ews$ or $\w \Ewu$,
then $l_i$ lifts homeomorphically to some leaf $l'_i$ of $\w \Fws$ or $\w \Fwu$;
if $l_i$ is a half-leaf of $\w \Ews$ or $\w \Ewu$,
then $\Int(l_i)$ lifts homeomorphically to $\Int(l'_i)$ for 
some half-leaf $l'_i$ of $\w \Fws$ or $\w \Fwu$.
In the latter case,
the half-leaves $l_i$ and $l'_i$ can be canonically identified
by identifying the closed orbits in their boundaries.

By Propositions~\ref{prop: l and Fws transverse}
and~\ref{prop: l and Fwu transverse},
every orbit of $\w \phi$ contained in $l_i$
intersects every leaf of $\w \F$ that intersects $l_i$.
Since $\sigma_1$ connects $o(y)$ to a point in
$\sigma_2 \subseteq o(l_2)$,
it follows that $\lambda$ intersects every orbit
contained in $l_2$.
Proceeding inductively along
$\sigma_2, \ldots, \sigma_k$,
we conclude that $\lambda$ intersects every orbit
contained in $l_k$.
In particular, $\lambda$ intersects $\w \phi(x)$.
\end{proof}

As an immediate consequence of Proposition~\ref{regulating},
we obtain the following.

\begin{prop}
The foliation $\F$ is $\R$-covered in $W$.
\end{prop}

\begin{proof}
Let $x \in \w W$.
By Proposition~\ref{regulating},
$\w \phi(x)$ intersects every leaf of $\w \F$.
Since $\w \phi$ is transverse to $\w \F$,
each leaf of $\w \F$ intersects $\w \phi(x)$ exactly once.
Therefore the leaf space $L(\F)$ can be canonically identified with
$\w \phi(x) \cong \R$,
and so $L(\F) \cong \R$.
\end{proof}

The proof of Theorem \ref{main} can be completed by
taking $\alpha^* = \alpha$.



\section{Foliations having at most one-sided branching}\label{sec: one-sided branching}

Throughout this section,
we adopt the hypotheses of Section~\ref{sec: R-covered},
namely,
$\Sigma$ is a compact orientable surface,
$\varphi \colon \Sigma \to \Sigma$
is an orientation-preserving pseudo-Anosov homeomorphism,
$M$ is the mapping torus of $\varphi$
with boundary components $T_1,\ldots,T_r$,
$\psi$ is the suspension flow of $\varphi$,
and $J_i$ is the interval of slopes on $T_i$
given in Theorem~\ref{foliation}.
For each $1 \leqslant i \leqslant r$, fix a slope
$s_i \in J_i \cap (\Q \cup \{\infty\})$,
and let $\s = (s_1,\ldots,s_r)$ be the corresponding multislope.

In contrast to Section \ref{sec: R-covered}, 
where a specific admissible system of arcs is chosen,
we now consider an arbitrary admissible system of arcs
and analyze the resulting foliation.
Let $\alpha=(\alpha_1,\ldots,\alpha_n)$ be
an admissible system of arcs with respect to $(\Sigma,\varphi)$.
Let $B_\alpha$ be the branched surface associated to $\alpha$
obtained in Construction~\ref{new branched surface}.
By Proposition~\ref{prop: construction from the new branched surface},
$B_\alpha$ fully carries a lamination $\Ll_\alpha$ that extends to
a foliation $\F_\alpha$ in $M$, 
which is transverse to $\psi$ and intersects each $T_i$
in simple closed curves of slope $s_i$.
Moreover, $\F_\alpha$ extends to a co-orientable taut foliation
$\F_\alpha(\s)$ in $M(\s)$.
Fried's surgery also produces a pseudo-Anosov flow $\phi$ on $M(\s)$
from $\psi$ that is transverse to $\F_\alpha(\s)$.
By Remark~\ref{rmk: only Fs},
Proposition~\ref{prop: l and Fws transverse}
still holds for $\F_\alpha$ and $\psi$,
since it only requires each $\alpha_i$ to be positively transverse to $\Fs$.

We prove Theorem~\ref{main} (b) in this section.

\begin{main (b)}
The foliation $\F_\alpha(\s)$ has at most one-sided branching.
\end{main (b)}

\subsection{The intersection behavior of $\w \F$ and $\w \Ewu$}\label{subsec: F and Ewu}

Let $W = M(\s)$ and $\F = \F_\alpha(\s)$.
We write $\Ews = \Fws(\phi)$ and $\Ewu = \Fwu(\phi)$ for
the weak stable and weak unstable foliations of $\phi$, respectively.
Let $\w W$ be the universal cover of $W$.
We write $\w \F, \w \phi, \w \Ews, \w \Ewu$
for the lifts of $\F, \phi, \Ews, \Ewu$ to $\w W$, respectively.
Let $q: \w W \to L(\F)$ denote the canonical quotient map which
projects every leaf of $\w \F$ to the corresponding point of $L(\F)$.

We prove the following proposition in this subsection.

\begin{prop}\label{prop: one-sided branching intersection}
Let $l$ be a regular leaf or a half-leaf of $\w{\Ewu}$.
For any two points $x,y \in l$ contained in distinct flowlines of $\w \phi$,
there exist $n_1, n_2 \in \R$ and a homeomorphism
\[h: (-\infty,n_1] \to (-\infty,n_2]\]
such that
\[q(\w \phi^{t}(x)) = q(\w \phi^{h(t)}(y))\]
for all $t \in (-\infty,n_1]$.
In particular,
$\w{\F} \mid_{l}$ has no branching in the negative direction.
\end{prop}
\begin{proof}
We fix a Riemannian metric $d(\cdot,\cdot)$ on $W$, 
and use the same notation for the induced metric on $\w W$.

\begin{figure}
	\centering
	\includegraphics[width=0.8\textwidth]{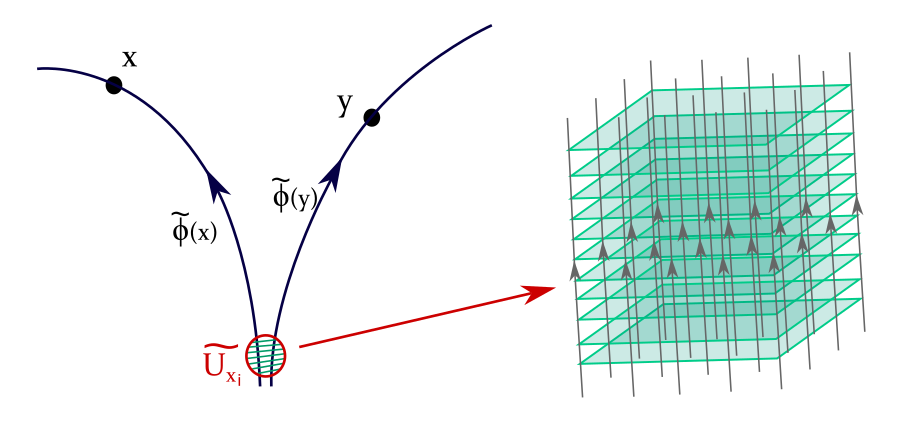}
	\caption{For $t \in \R$ sufficiently negative,
a lift $\w{U_{x_i}}$ of $U_{x_i}$ contains $\w\phi^{t}(x)$
and intersects $\w\phi(y)$;
hence the leaf of $\w\F$ intersecting $\w\phi^{t}(x)$
also meets $\w\phi(y)$.}\label{one-sided branching on Ewu}
\end{figure}

Since $\phi$ is transverse to $\F$, for each $x \in W$ there exists 
an open neighborhood $U_x \cong \R^3$ of $x$ and 
a homeomorphism $h: U_x \to \R^2 \times \R$ such that
\[h(\F \mid_{U_x}) = \{(\R^2,v) \mid v \in \R\},\]
\[h(\phi \mid_{U_x}) = \{(u,\R) \mid u \in \R^2\}.\]
Since $W$ is compact, 
there exists a finite set $P = \{x_1, \ldots, x_m\} \subseteq W$ such that
\[\bigcup_{i=1}^{m} U_{x_i} = W.\]
Since $\{U_{x_i}\}_{i=1}^m$ is a finite open cover of the compact metric space $W$,
by the Lebesgue number lemma,
there exists a constant $c > 0$ such that
for any $x \in W$, the $c$-ball $\{ y \in W \mid d(x,y) < c \}$ 
is contained in some $U_{x_i}$.

Let $x,y \in l$ be contained in two distinct flowlines of $\w \phi$.
There exists $n \in \R$ sufficiently negative such that
\[d(\w \phi^{t}(x), \w \phi(y)) < c, \quad \text{for all } t < n,\]
\[d(\w \phi(x), \w \phi^{t}(y)) < c, \quad \text{for all } t < n.\]
Let $t < n$, and let $\lambda$ be the leaf of $\w \F$ intersecting $\w \phi^{t}(x)$.
As $\w \phi(y)$ meets the $c$-ball \[\{s \in \w W \mid d(s,\w \phi^{t}(x)) < c\},\]
there exists a lift $\w{U_{x_i}}$ of some $U_{x_i}$ to $\w W$
intersected by both $\lambda$ and $\w \phi(y)$.
It follows that $\lambda \cap \w \phi(y) \ne \emptyset$.
This implies that every leaf of $\w \F$ intersecting 
$\w \phi^{(-\infty,n]}(x)$ also intersects $\w \phi(y)$.
Similarly, every leaf of $\w \F$ intersecting $\w \phi^{(-\infty,n]}(y)$
also intersects $\w \phi(x)$.

Let $n_1,n_2 \in \R$ be such that 
$n_1,n_2 < n$ and
$q(\w \phi^{n_1}(x)) = q(\w \phi^{n_2}(y))$.
Then
\[q(\w \phi^{(-\infty,n_1]}(x)) = q(\w \phi^{(-\infty,n_2]}(y)).\]
Since $\w \phi$ is transverse to $\w \F$,
there exists a homeomorphism $h: (-\infty,n_1] \to (-\infty,n_2]$ such that
\[q(\w \phi^{t}(x)) = q(\w \phi^{h(t)}(y))\]
for all $t \in (-\infty,n_1]$.

Let $\lambda_1, \lambda_2$ be two arbitrary distinct leaves of $\w \F$ intersecting $l$.
Choose $u \in \lambda_1 \cap l$ and $v \in \lambda_2 \cap l$.
By the above conclusion,
there exist $t_1,t_2 \in \R$ sufficiently negative such that
$q(\w \phi^{t_1}(u)) = q(\w \phi^{t_2}(v))$.
Let $\mu \in L(\F)$ denote the leaf $q(\w \phi^{t_1}(u))$.
Then $\mu < \lambda_1$ and $\mu < \lambda_2$.
It follows that $\w \F \mid_l$ has no branching in the negative direction.
\end{proof}

\subsection{Branching behavior of $\F$}\label{subsec: verify one-sided}

In this subsection,
we complete the proof of Theorem \ref{main} (b).

\begin{prop}
    The foliation $\F$ either has one-sided branching or is $\R$-covered.
\end{prop}
\begin{proof}
Let $\lambda_1, \lambda_2$ be two arbitrary distinct leaves of $\w \F$.
We show that there exists a leaf $\mu$ of $\w \F$ such that
$\mu \Lle \lambda_1$ and $\mu \Lle \lambda_2$.

Choose points $x,y \in \w W$ with
$x \in \lambda_1$ and $y \in \lambda_2$.
Consider the canonical projection
$o: \w W \to \Or(\phi)$.
As in the proof of Proposition~\ref{regulating},
there exists a broken path
\[\sigma = \sigma_1 * \sigma_2 * \ldots * \sigma_k : I \to \Or(\phi)\]
such that $\sigma(0) = o(y)$, $\sigma(1) = o(x)$,
and each segment $\sigma_i$ is contained in a leaf of
$\Ors(\phi)$ or $\Oru(\phi)$.

Let $C$ denote the union of closed orbits of $\phi$
obtained by collapsing the filling solid tori in Fried's surgery,
and let $\w C$ be its pullback to $\w W$.
Arguing exactly as in the proof of Proposition~\ref{regulating},
after subdividing the segments if necessary,
we may assume that for every $i$,
the interior $\Int(\sigma_i)$ is disjoint from
the singularities of $\Or(\phi)$ and from the set $o(\w C)$.
As before, there exists a regular leaf or half-leaf $l_i$ of
$\w \Ews$ or $\w \Ewu$ such that
\[\sigma_i \subseteq o(l_i), \quad
\Int(l_i) \subseteq \w W - \w C.\]

Fix an orientation on $\sigma$ from $o(y)$ to $o(x)$,
which induces an orientation on each $\sigma_i$.
Choose points $u_i, v_i \in \w W$ for each $i$ such that
$o(u_i), o(v_i) \in \partial \sigma_i$ and
the induced orientation on $\sigma_i$ runs from $o(u_i)$ to $o(v_i)$.

We claim that for any leaf $\lambda$ intersecting $\w \phi(u_i)$,
there exists a leaf $\lambda'$ with $\lambda' \Lle \lambda$
that intersects $\w \phi(v_i)$.
If $l_i$ is a regular leaf or half-leaf of $\w \Ewu$,
this is an immediate consequence of
Proposition~\ref{prop: one-sided branching intersection}.
Now suppose that $l_i$ is a regular leaf or half-leaf of $\w \Ews$.
As in the proof of Proposition~\ref{regulating},
$l_i$ can be canonically identified with a regular leaf or half-leaf of 
the pullback foliation of $\Fws(\psi)$ to $\w M$.
Since Proposition~\ref{prop: l and Fws transverse} holds for $\F_\alpha$ and $\psi$,
every orbit of $\w \phi$ contained in $l_i$
intersects every leaf of $\w \F$ meeting $l_i$.
In particular, $\w \phi(v_i)$ intersects every leaf of $\w \F$
that intersects $\w \phi(u_i)$.
This completes the proof of the claim.

Starting from $\lambda_2$,
choose a leaf $\mu_1 \Lle \lambda_2$ that intersects $\w \phi(v_1)$.
Since $\w \phi(u_2) = \w \phi(v_1)$,
we can then choose a leaf $\mu_2 \Lle \mu_1$ that intersects $\w \phi(v_2)$.
Proceeding inductively, we obtain leaves
\[\mu_k \Lle \cdots \Lle \mu_1 \Lle \lambda_2\]
such that $\mu_k$ intersects $\w \phi(v_k)=\w \phi(x)$.
Since $\mu_k$ and $\lambda_1$ both intersect $\w \phi(x)$,
we can choose a leaf $\mu$ intersecting $\w \phi(x)$ with
\[\mu \Lle \mu_k, \lambda_1.\]
In particular, $\mu \Lle \mu_k \Lle \lambda_2$ and $\mu \Lle \lambda_1$, as desired.

By Definition~\ref{one-sided branching},
$\F$ has one-sided branching (in the positive direction)
if $L(\F)$ is not homeomorphic to $\R$.
Otherwise, $\F$ is $\R$-covered.
This completes the proof.
\end{proof}

\bibliographystyle{spmpsci}
\bibliography{refs}

\end{document}